\documentclass[11pt]{article}
\usepackage{amsmath,amssymb,amsthm}
\usepackage{mathtools}
\usepackage{enumerate}
\usepackage[margin=1in]{geometry}
\usepackage{hyperref}
\hypersetup{
    colorlinks=true,
    linkcolor=blue,
    citecolor=blue,
    urlcolor=blue,
    pdfauthor={Paolo Vella},
    pdftitle={Defect Cocycles and the Structure of Finite Process Monoids}
}

\newtheorem{theorem}{Theorem}[section]
\newtheorem{lemma}[theorem]{Lemma}
\newtheorem{proposition}[theorem]{Proposition}
\newtheorem{corollary}[theorem]{Corollary}

\theoremstyle{definition}
\newtheorem{definition}[theorem]{Definition}
\newtheorem{example}[theorem]{Example}
\theoremstyle{remark}
\newtheorem{remark}[theorem]{Remark}

\DeclareMathOperator{\supp}{supp}
\DeclareMathOperator{\Ran}{Ran}
\DeclareMathOperator{\rank}{rank}

\title{Defect Cocycles and the Structure of Finite Process Monoids}
\author{Paolo Vella}
\date{}

\begin{document}

\maketitle

\begin{abstract}
We study positive subunital maps on ordered effect spaces and introduce the defect $d(T) = u - T(u)$, which satisfies a cocycle identity under composition. Using only this identity and elementary order-theoretic arguments---requiring no spectral decomposition or dimension-dependent techniques---we prove that in any finite composition-closed family of positive subunital maps, defects are eventually annihilated under iteration (Theorem~4.1), with an explicit bound linear in the family size. Under a persistence hypothesis (nonzero positive elements map to nonzero positive elements), we establish that all maps in such families must be unital. For completely positive maps on finite-dimensional matrix algebras, we then prove a sharp dimension-dependent bound: the stabilization index satisfies $n_T \le d$ where $d$ is the Hilbert space dimension, independent of the family size. This bound is achieved by a shift channel construction. These results provide a structural explanation for why finite operational repertoires in process theories cannot sustain systematic information loss, with applications to quantum foundations and categorical probability.
\end{abstract}

\section{Introduction}

In operational approaches to quantum theory and categorical probability, physical processes are modeled as positive maps preserving certain algebraic and order structures. A central question is: what constraints does the finiteness of an operational repertoire impose on the processes it can contain?

This paper introduces a simple but powerful tool---the defect cocycle---and uses it to prove that finite collections of processes closed under composition cannot sustain systematic ``leakage.'' More precisely, if $A$ is a finite set of positive subunital maps closed under composition, then for every $T \in A$, the defect $d(T) = u - T(u)$ is eventually annihilated under iteration: there exists $n$ with $T^n(d(T)) = 0$. Operationally, this means that after $n$ iterations of $T$, additional iterations no longer reduce the success probability---the process has ``absorbed all the loss it will ever absorb.''

The proof is entirely elementary, relying only on:
\begin{enumerate}[(i)]
\item the cocycle identity $d(T \circ S) = d(T) + T(d(S))$,
\item positivity and the pointed cone condition,
\item finiteness of $A$.
\end{enumerate}
No spectral theory, no dimension bounds, no continuity assumptions. This ``spectral-free'' character makes the result applicable across diverse settings: finite-dimensional quantum operations (including channels), stochastic matrices, and abstract effectus theory.

Under an additional persistence hypothesis (nonzero positive elements stay nonzero), we upgrade the conclusion to unitality: every $T \in A$ must satisfy $d(T) = 0$, meaning $T(u) = u$. This connects to the physical intuition that processes which cannot ``lose'' probability mass in a detectable way must preserve normalization.

For completely positive maps on matrix algebras $M_d$, we prove a sharp dimension bound: $n_T \le d$, independent of the family size $|A|$. The proof proceeds by showing that the defect orbit generates a corner on which the dynamics is genuinely nilpotent, then applying an elementary linear-algebra bound for nilpotent CP maps (via Kraus words and a strict kernel filtration). A shift channel construction shows this bound is optimal.

\paragraph{Clarification: finite composition-closed families.} A key hypothesis throughout is that maps belong to a \emph{finite composition-closed} family $A$. This means $A$ is finite and $S, T \in A$ implies $S \circ T \in A$. A crucial consequence: if $T \in A$, then $\{T, T^2, T^3, \ldots\} \subseteq A$, so this set is finite. The orbit $\{T^k(u)\}$ must therefore eventually repeat, which is the engine driving stabilization. Readers should not confuse this with the trivial observation that a single map forms a finite set; the composition-closure forces \emph{orbital finiteness}.

\paragraph{Related work.} The defect cocycle is implicit in sequential effect algebras \cite{gudder-greechie} and categorical treatments of kernels in effectus theory \cite{cho-jacobs}. Our contribution is to isolate the cocycle identity as the key mechanism and to derive strong structural consequences from finiteness alone.

Theorem~\ref{thm:defect-annihilation} resembles the classical result that every element of a finite semigroup has an idempotent power \cite{howie}. However, the mechanisms differ: in abstract semigroups, finiteness forces eventual periodicity, while here the interplay of positivity and order structure forces \emph{stabilization} rather than mere periodicity. The defect cocycle converts the algebraic pigeonhole argument into a statement about annihilation in the ordered cone.

\paragraph{Outline.} Section~\ref{sec:prelim} establishes notation and the basic properties of ordered effect spaces. Section~\ref{sec:cocycle} introduces the defect and proves the cocycle identity. Section~\ref{sec:main} contains the main theorem and its corollaries. Section~\ref{sec:examples} provides examples and counterexamples. Section~\ref{sec:parallel} develops the theory of defects under parallel composition. Section~\ref{sec:asymptotic} extends the theory to infinite dimensions via asymptotic defects. Section~\ref{sec:quantum} establishes dimension-dependent bounds for quantum operations, including the sharp bound $n_T \le d$ for CP maps. Section~\ref{sec:discussion} discusses applications and open problems.

%=============================================================================
\section{Preliminaries}\label{sec:prelim}
%=============================================================================

\begin{definition}[Ordered effect space]
An \emph{ordered effect space} is a triple $(E, \le, u)$ where:
\begin{enumerate}[(i)]
\item $E$ is a real vector space,
\item $\le$ is a preorder on $E$ (reflexive and transitive) compatible with the vector space structure: $x \le y$ implies $x + z \le y + z$ and $\lambda x \le \lambda y$ for all $z \in E$ and $\lambda \ge 0$,
\item $u \in E_+$ is a distinguished positive element called the \emph{unit}.
\end{enumerate}
The \emph{positive cone} is $E_+ = \{x \in E : x \ge 0\}$, and the \emph{effect interval} is $[0, u] = \{x \in E : 0 \le x \le u\}$.
\end{definition}

\begin{definition}[Pointed cone]
The positive cone $E_+$ is \emph{pointed} if $E_+ \cap (-E_+) = \{0\}$, i.e., $x \ge 0$ and $-x \ge 0$ together imply $x = 0$.
\end{definition}

\begin{remark}[Relationship between pointedness and antisymmetry]\label{rem:pointed-antisymmetry}
If $\le$ is a \emph{partial order} (antisymmetric), then pointedness is automatic: $x \ge 0$ and $-x \ge 0$ imply $0 \le x \le 0$, hence $x = 0$ by antisymmetry. Conversely, if $E_+$ is pointed, then $\le$ is antisymmetric. Thus pointedness is a nontrivial hypothesis only when $\le$ is a proper preorder (not antisymmetric). We state it explicitly because some categorical or effectus-theoretic settings work with preorders, and our results apply whenever the cancellation property (Lemma~\ref{lem:cancellation}) holds.
\end{remark}

\begin{remark}
We do not require $u$ to be an order unit in the Archimedean sense; we only need a fixed $u \in E_+$ to define subunitality. Many readers associate ``effect space'' with an order-unit assumption, but our results hold in this more general setting.
\end{remark}

\begin{remark}[Why pointedness is necessary]\label{rem:pointed-necessary}
The pointed cone condition is essential for our main results. Without it, the cancellation lemma below fails. For a concrete counterexample, consider $E = \mathbb{R}^2$ with the non-pointed cone $E_+ = \{(x,y) : x \ge 0\}$ (a half-plane), which induces a preorder that is not antisymmetric. Then $(1, -1)$ and $(-1, 1)$ are both in $E_+$, and their sum is zero without either being zero.
\end{remark}

\begin{lemma}[Pointed cone cancellation]\label{lem:cancellation}
If $E_+$ is pointed and $x_1, \ldots, x_m \in E_+$ satisfy $x_1 + \cdots + x_m = 0$, then $x_i = 0$ for all $i$.

Equivalently: $E_+$ is pointed if and only if for all $x, y \ge 0$, $x + y = 0$ implies $x = y = 0$.
\end{lemma}

\begin{proof}
We have $x_1 = -(x_2 + \cdots + x_m)$. Since each $x_j \ge 0$, the sum $x_2 + \cdots + x_m \in E_+$, so $x_1 \in -E_+$. But $x_1 \in E_+$ by assumption, hence $x_1 \in E_+ \cap (-E_+) = \{0\}$. Induction on $m$ completes the proof. (The equivalence is immediate: $(\Rightarrow)$ is the $m = 2$ case; $(\Leftarrow)$ if $x \in E_+ \cap (-E_+)$ then $x + (-x) = 0$ with $x, -x \ge 0$, so $x = 0$.)
\end{proof}

\begin{definition}[Positive and subunital maps]
Let $(E, \le, u)$ be an ordered effect space. A linear map $T : E \to E$ is:
\begin{enumerate}[(i)]
\item \emph{positive} if $T(E_+) \subseteq E_+$, i.e., $x \ge 0$ implies $T(x) \ge 0$;
\item \emph{subunital} if $T(u) \le u$;
\item \emph{unital} if $T(u) = u$.
\end{enumerate}
\end{definition}

\begin{lemma}[Monotonicity]
If $T$ is positive, then $T$ is order-preserving: $x \le y$ implies $T(x) \le T(y)$.
\end{lemma}

\begin{proof}
If $x \le y$, then $y - x \ge 0$, so $T(y - x) = T(y) - T(x) \ge 0$ by positivity.
\end{proof}

%=============================================================================
\section{The Defect Cocycle}\label{sec:cocycle}
%=============================================================================

\begin{definition}[Defect]
For a subunital map $T$ on $(E, \le, u)$, the \emph{defect} of $T$ is
\[
d(T) = u - T(u).
\]
\end{definition}

By subunitality, $T(u) \le u$, so $d(T) = u - T(u) \ge 0$. The defect measures the ``shortfall'' of $T(u)$ from the unit---physically, the probability of termination, the no-click outcome, or the information lost to the environment.

\begin{lemma}[Cocycle identity]\label{lem:cocycle}
For positive subunital maps $T, S$ on $(E, \le, u)$,
\[
d(T \circ S) = d(T) + T(d(S)).
\]
\end{lemma}

\begin{proof}
Direct computation:
\begin{align*}
d(T \circ S) &= u - T(S(u)) \\
&= u - T(u) + T(u) - T(S(u)) \\
&= d(T) + T(u - S(u)) \\
&= d(T) + T(d(S)). \qedhere
\end{align*}
\end{proof}

\begin{lemma}[Iterated cocycle]\label{lem:iterated-cocycle}
For a positive subunital map $T$ and $n \ge 1$,
\[
d(T^n) = \sum_{k=0}^{n-1} T^k(d(T)).
\]
Equivalently,
\[
u - T^n(u) = \sum_{k=0}^{n-1} T^k(d(T)).
\]
\end{lemma}

\begin{proof}
By induction on $n$. The case $n = 1$ is immediate: $d(T^1) = d(T) = T^0(d(T))$.

For the inductive step, assume the formula holds for $n$. Then by Lemma~\ref{lem:cocycle},
\begin{align*}
d(T^{n+1}) &= d(T \circ T^n) = d(T) + T(d(T^n)) \\
&= d(T) + T\left(\sum_{k=0}^{n-1} T^k(d(T))\right) \\
&= d(T) + \sum_{k=0}^{n-1} T^{k+1}(d(T)) \\
&= T^0(d(T)) + \sum_{k=1}^{n} T^k(d(T)) \\
&= \sum_{k=0}^{n} T^k(d(T)). \qedhere
\end{align*}
\end{proof}

\begin{remark}[Monotonicity of iterates]\label{rem:monotone}
Subunitality and positivity together imply that the sequence $\{T^k(u)\}_{k \ge 0}$ is monotone decreasing. Indeed, $T(u) \le u$ by subunitality, and applying the positive (hence order-preserving) map $T$ to both sides yields $T^2(u) \le T(u)$. By induction,
\[
u \ge T(u) \ge T^2(u) \ge T^3(u) \ge \cdots
\]
In any setting where the orbit is finite, such monotone sequences must stabilize. This observation provides immediate intuition for Theorem~\ref{thm:defect-annihilation}.
\end{remark}

%=============================================================================
\section{Main Results}\label{sec:main}
%=============================================================================

\begin{theorem}[Defect annihilation]\label{thm:defect-annihilation}
Let $(E, \le, u)$ be an ordered effect space with pointed positive cone. Let $A$ be a finite set of positive subunital maps on $E$ that is closed under composition. Then for each $T \in A$, there exists $n \ge 1$ such that
\[
T^n(d(T)) = 0.
\]
Moreover, the stabilization index satisfies $n \le |A|$.
\end{theorem}

\begin{proof}
Since $A$ is closed under composition and $T \in A$, we have $T^k \in A$ for all $k \ge 1$. Therefore, the set
\[
\{T^k(u) : k \ge 1\} \subseteq \{S(u) : S \in A\}
\]
is finite, with cardinality at most $|A|$.

By the pigeonhole principle, there exist integers $m > n \ge 1$ with $T^m(u) = T^n(u)$. Set $p = m - n \ge 1$. Then
\[
T^n(u) = T^m(u) = T^{n+p}(u) = T^n(T^p(u)),
\]
which gives $T^n(u - T^p(u)) = 0$.

By the iterated cocycle identity (Lemma~\ref{lem:iterated-cocycle}),
\[
u - T^p(u) = \sum_{k=0}^{p-1} T^k(d(T)).
\]
Each term $T^k(d(T)) \ge 0$ by positivity of $T$ and $d(T) \ge 0$. Applying the positive map $T^n$ to the sum, we have
\[
T^n(u - T^p(u)) = \sum_{k=0}^{p-1} T^{n+k}(d(T)) = 0.
\]
By Lemma~\ref{lem:cancellation} (pointed cone cancellation), since each $T^{n+k}(d(T)) \ge 0$, we obtain
\[
T^{n+k}(d(T)) = 0 \quad \text{for all } k \in \{0, 1, \ldots, p-1\}.
\]
In particular, taking $k = 0$ yields $T^n(d(T)) = 0$.

For the bound: the orbit $\{T^k(u) : k \ge 1\}$ has at most $|A|$ elements. Among the $|A| + 1$ iterates $T^1(u), T^2(u), \ldots, T^{|A|+1}(u)$, there must be a repeat by pigeonhole. If $T^m(u) = T^n(u)$ with $m > n \ge 1$, then $n \le |A|$.
\end{proof}

\begin{remark}[Stabilization]\label{rem:stabilization}
The conclusion $T^n(d(T)) = 0$ is equivalent to $T^{n+1}(u) = T^n(u)$. Indeed, by the cocycle identity, $T^{n+1}(u) = T^n(T(u)) = T^n(u - d(T)) = T^n(u) - T^n(d(T))$. Hence $T^n(d(T)) = 0$ iff $T^{n+1}(u) = T^n(u)$, which then implies $T^{n+k}(u) = T^n(u)$ for all $k \ge 0$.
\end{remark}

\begin{definition}[Stabilization index]\label{def:stabilization-index}
For a positive subunital map $T$ in a finite composition-closed family, the \emph{stabilization index} is
\[
n_T := \min\{n \ge 1 : T^n(d(T)) = 0\}.
\]
By Theorem~\ref{thm:defect-annihilation}, $n_T$ is well-defined and satisfies $n_T \le |A|$.
\end{definition}

\begin{remark}[Operational interpretation of stabilization]\label{rem:operational}
The condition $T^n(d(T)) = 0$ admits a clear process-theoretic interpretation. The following are equivalent:
\begin{enumerate}[(i)]
\item $T^n(d(T)) = 0$;
\item $T^{n+1}(u) = T^n(u)$ (unit orbit stabilizes);
\item $T^{n+k}(u) = T^n(u)$ for all $k \ge 0$ (unit orbit becomes constant);
\item the ``loss per additional iteration'' $T^k(d(T))$ vanishes for $k \ge n$.
\end{enumerate}
Operationally: after $n$ iterations, further iterations do not reduce the success probability. The process has ``absorbed all the loss it will ever absorb'' into a stable corner.
\end{remark}

\begin{remark}[Comparison with finite semigroup theory]\label{rem:semigroup}
In any finite semigroup, powers eventually become \emph{periodic}: there exist $n, p \ge 1$ with $T^{n+p} = T^n$. In contrast, our setting (positive subunital maps on an ordered space) yields \emph{stabilization}: the unit orbit becomes \emph{constant}, not merely periodic. The key is that subunitality makes $\{T^k(u)\}$ monotone decreasing (Remark~\ref{rem:monotone}), so periodicity of the orbit forces constancy. The defect cocycle makes this bookkeeping explicit and provides the annihilation bound.
\end{remark}

\begin{remark}[Alternative proof via monotonicity]\label{rem:alt-proof}
Theorem~\ref{thm:defect-annihilation} can also be proved without the cocycle expansion, using only monotonicity and finiteness. Since $\{T^k(u)\}$ is monotone decreasing (Remark~\ref{rem:monotone}) and $\{T^k(u) : k \ge 1\} \subseteq \{S(u) : S \in A\}$ is finite, the sequence must stabilize: there exists $n$ with $T^n(u) = T^{n+1}(u)$, hence $T^n(d(T)) = T^n(u) - T^{n+1}(u) = 0$. The cocycle proof given above has the advantage of explicitly exhibiting the defect decomposition and motivating the corner/nilpotency analysis in Section~\ref{sec:quantum}.

\textbf{Warning (preorder subtlety):} This ``monotonicity-only'' argument implicitly uses \emph{antisymmetry} (equivalently pointedness). In a proper preorder (not antisymmetric), monotone orbits can be eventually periodic without becoming constant---e.g., cycling between distinct preorder-equivalent elements. The cocycle proof via Lemma~\ref{lem:cancellation} is valid for preorders precisely because it uses pointedness directly to force the summands to vanish.
\end{remark}

\begin{remark}[On the pointedness hypothesis]\label{rem:pointed}
In the proof of Theorem~\ref{thm:defect-annihilation}, pointedness is used only in Lemma~\ref{lem:cancellation} to conclude that a sum of positive elements equaling zero implies each summand is zero. As noted in Remark~\ref{rem:pointed-antisymmetry}, pointedness is equivalent to antisymmetry of $\le$, so the hypothesis is nontrivial only when $\le$ is a proper preorder. For most applications (matrix algebras, stochastic maps, etc.), $\le$ is a partial order and pointedness is automatic.
\end{remark}

\begin{remark}[Where pointedness is used]\label{rem:pointedness-map}
For the reader's convenience, we summarize where the pointedness hypothesis appears in the paper:
\begin{itemize}
\item \textbf{Lemma~\ref{lem:cancellation}} (pointed cone cancellation): the fundamental engine.
\item \textbf{Theorem~\ref{thm:defect-annihilation}}: via Lemma~\ref{lem:cancellation} to conclude summands vanish.
\item \textbf{Remark~\ref{rem:alt-proof}}: the ``monotonicity-only'' alternative proof implicitly uses antisymmetry (equivalent to pointedness); without it, monotone orbits can cycle.
\item \textbf{Theorem~\ref{thm:abstract-stabilization}} (Section~\ref{sec:discussion}): the categorical axiomatization explicitly assumes pointed cone.
\item \textbf{Section~\ref{sec:quantum}}: in finite-dimensional matrix algebras, the PSD cone is pointed (automatic), so pointedness is implicitly satisfied.
\item \textbf{Section~\ref{sec:asymptotic}}: von Neumann algebra settings use faithful traces/states, which force pointedness (if $\omega(x) = 0$ for positive $x$, then $x = 0$).
\end{itemize}
In short: pointedness is only a nontrivial hypothesis for exotic preordered vector spaces; all standard operator-algebraic settings satisfy it automatically.
\end{remark}

\begin{corollary}[Explicit stabilization bound]\label{cor:bound}
Under the hypotheses of Theorem~\ref{thm:defect-annihilation}, the minimal $n$ such that $T^n(d(T)) = 0$ satisfies $n \le |A|$. If $A$ is a submonoid (i.e., contains the identity), then $n \le |A| - 1$.
\end{corollary}

\begin{proof}
The first bound follows from the proof of Theorem~\ref{thm:defect-annihilation}. If $A$ contains the identity, then the orbit $\{T^k(u) : k \ge 0\}$ includes $T^0(u) = u$, so the orbit has at most $|A|$ elements starting from $k = 0$, and collision occurs by step $|A|$, yielding $n \le |A| - 1$.
\end{proof}

\begin{remark}
The bound $n \le |A|$ is sharp in general. Consider $A = \{T, T^2, \ldots, T^{m-1}, 0\}$ where $T$ is nilpotent with $T^m = 0$ but $T^{m-1} \ne 0$ (composition-closed since $T^m = 0$). Then $|A| = m$ and the minimal $n$ with $T^n(d(T)) = 0$ is exactly $m$.
\end{remark}

\subsection*{Persistence and unitality}

\begin{definition}[Persistence]\label{def:persistence}
A positive map $T : E \to E$ is \emph{persistent} if
\[
x > 0 \implies T(x) > 0,
\]
where $x > 0$ means $x \in E_+ \setminus \{0\}$ (nonzero positive). Equivalently: $T$ does not annihilate any nonzero positive element.
\end{definition}

\begin{remark}[Persistence in standard settings]\label{rem:persistence-standard}
\emph{Caution}: Our ``persistent'' (nonzero positive $\mapsto$ nonzero positive) is weaker than the operator-algebraic notion of ``strictly positive'' or ``positivity-improving'' ($\mathrm{int}(E_+) \to \mathrm{int}(E_+)$). We use ``persistent'' to avoid confusion.

For substochastic matrices on $\mathbb{R}^n$ with componentwise order, persistence means all entries are strictly positive: $P_{ij} > 0$ for all $i, j$. This is stronger than irreducibility. For quantum channels in the Heisenberg picture $\Phi^*(Y) = \sum_i V_i^* Y V_i$, persistence holds whenever $\mathrm{span}\{V_i\} = M_d$ (a sufficient condition is that some $V_j$ is invertible); see \cite{wolf} for related positivity-improving conditions.
\end{remark}

\begin{theorem}[Persistence implies unitality]\label{thm:persistence-unitality}
Let $(E, \le, u)$ be an ordered effect space with pointed positive cone. Let $A$ be a finite set of positive subunital maps on $E$ that is closed under composition. If every $T \in A$ is persistent, then every $T \in A$ is unital.
\end{theorem}

\begin{proof}
By Theorem~\ref{thm:defect-annihilation}, for each $T \in A$ there exists $n \ge 1$ with $T^n(d(T)) = 0$.

If $d(T) > 0$, then by persistence of $T$, we have $T(d(T)) > 0$, hence $T^2(d(T)) > 0$, and by induction $T^n(d(T)) > 0$ for all $n \ge 0$. This contradicts $T^n(d(T)) = 0$.

Therefore $d(T) = 0$, i.e., $T(u) = u$.
\end{proof}

\begin{corollary}
A finite submonoid of persistent positive subunital maps consists entirely of unital maps.
\end{corollary}

\subsection*{Corner-faithfulness: a weaker condition for unitality}

The persistence condition is quite strong. We now introduce a weaker condition that still implies unitality in finite composition-closed families.

\begin{definition}[Corner-faithfulness]\label{def:corner-faithful}
Let $T$ be a positive subunital map on $M_d$ with finite stabilization index $n_T < \infty$ and orbit-support projection $Q_T := \bigvee_{k < n_T} \supp(T^k(d(T)))$. We say $T$ is \emph{corner-faithful} if
\[
x \ge 0, \quad x \ne 0, \quad \supp(x) \le Q_T \implies T(x) \ne 0.
\]
\end{definition}

\begin{theorem}[Corner-faithfulness implies unitality]\label{thm:corner-faithful-unital}
Let $T$ be a CP subunital map on $M_d$ with finite stabilization index $n_T < \infty$. If $T$ is corner-faithful, then $T$ is unital.
\end{theorem}

\begin{proof}
Suppose $d(T) \ne 0$. By minimality of $n_T$, we have $T^{n_T - 1}(d(T)) \ne 0$ while $T^{n_T}(d(T)) = 0$. Since $\supp(T^{n_T - 1}(d(T))) \le Q_T$, corner-faithfulness implies $T(T^{n_T - 1}(d(T))) \ne 0$, contradicting $T^{n_T}(d(T)) = 0$.
\end{proof}

\begin{remark}
The argument only needs existence of $n$ with $T^n(d(T)) = 0$; membership in a finite composition-closed family is one sufficient hypothesis (Theorem~\ref{thm:defect-annihilation}).
\end{remark}

\begin{proposition}[Kraus criterion for corner-faithfulness]\label{prop:corner-faithful-kraus}
Let $T(X) = \sum_i V_i^* X V_i$ be CP on $M_d$ with orbit-support projection $Q$. Then $T$ is corner-faithful iff
\[
QH \cap \bigcap_i \ker(V_i^*) = \{0\}.
\]
Equivalently: for every $0 \ne v \in QH$, there exists $i$ with $V_i^* v \ne 0$.
\end{proposition}

\begin{proof}
For rank-one $x = vv^*$ with $v \in QH$:
\[
T(vv^*) = \sum_i V_i^* vv^* V_i = \sum_i |V_i^* v\rangle\langle V_i^* v|,
\]
which vanishes iff $V_i^* v = 0$ for all $i$. Since every positive $x$ with $\supp(x) \le Q$ is a sum of such rank-ones, corner-faithfulness holds iff no nonzero $v \in QH$ lies in $\bigcap_i \ker(V_i^*)$.
\end{proof}

\begin{remark}[Hierarchy of unitality conditions]\label{rem:faithfulness-hierarchy}
The hierarchy of conditions ensuring unitality in finite composition-closed families is:
\[
\text{persistence (Thm~\ref{thm:persistence-unitality})} \implies \text{corner-faithfulness} \implies \text{unitality}.
\]
Persistence requires $T$ to map nonzero positive elements to nonzero positive elements; corner-faithfulness only requires this for elements supported on the orbit corner $Q_T$. The Kraus criterion shows corner-faithfulness is equivalent to $\{V_i^*\}$ having trivial common kernel on $QH$.
\end{remark}

\begin{proposition}[Projection-faithfulness implies unitality]\label{prop:projection-faithful}
Let $T : M_d \to M_d$ be CP subunital with finite stabilization index $n_T < \infty$. Suppose:
\[
\text{for every nonzero projection } 0 \ne p \le Q_T, \text{ we have } T(p) \ne 0. \quad \text{(projection-faithfulness)}
\]
Then $T$ is corner-faithful, hence unital.
\end{proposition}

\begin{proof}
Let $0 \ne x \ge 0$ with $\supp(x) \le Q_T$. Write the spectral decomposition $x = \sum_j \lambda_j p_j$ with $\lambda_j > 0$ and $0 \ne p_j \le Q_T$. By projection-faithfulness, $T(p_j) \ne 0$ and $T(p_j) \ge 0$ for each $j$. Hence $T(x) = \sum_j \lambda_j T(p_j) \ne 0$. Thus $T$ is corner-faithful, and unitality follows from Theorem~\ref{thm:corner-faithful-unital}.
\end{proof}

\begin{remark}[Projection-faithfulness $\Leftrightarrow$ Kraus common-kernel condition]
By Proposition~\ref{prop:corner-faithful-kraus}, projection-faithfulness is equivalent to
\[
Q_T H \cap \bigcap_i \ker(V_i^*) = \{0\}.
\]
Indeed, if a nonzero $v \in Q_T H$ lies in the common kernel, then $p = |v\rangle\langle v| \le Q_T$ satisfies $T(p) = 0$. Conversely, if $T(p) = 0$ for some nonzero projection $p \le Q_T$, the range of $p$ contains such a $v$. Thus projection-faithfulness and corner-faithfulness are equivalent for CP maps.
\end{remark}

%=============================================================================
\section{Examples and Counterexamples}\label{sec:examples}
%=============================================================================

\begin{example}[Scalar case]\label{ex:scalar}
Let $E = \mathbb{R}$ with cone $E_+ = [0, \infty)$ and unit $u = 1$. A positive subunital map is multiplication by some $\lambda \in [0, 1]$, with defect $d(T_\lambda) = 1 - \lambda$.

For $\lambda \in (0, 1)$, we have $T_\lambda^n(d(T_\lambda)) = \lambda^n(1 - \lambda)$, which is never zero for finite $n$. This shows the finiteness hypothesis in Theorem~\ref{thm:defect-annihilation} is essential.

However, if we take $A = \{T_0, T_1\}$ (the only finite composition-closed sets of such maps are $\{T_0\}$, $\{T_1\}$, and $\{T_0, T_1\}$), then indeed $T_0(d(T_0)) = 0$ and $d(T_1) = 0$.
\end{example}

\begin{example}[Defect-annihilating but nonunital]\label{ex:projection}
Let $E = \mathbb{R}^2$ with standard cone $E_+ = \{(x, y) : x, y \ge 0\}$ and unit $u = (1, 1)$. Consider the projection
\[
T(x, y) = (x, 0).
\]
Then $T$ is positive, subunital (since $T(u) = (1, 0) \le (1, 1) = u$), and $d(T) = (0, 1)$. We have
\[
T(d(T)) = T(0, 1) = (0, 0) = 0,
\]
so $T^1(d(T)) = 0$, satisfying the conclusion of Theorem~\ref{thm:defect-annihilation} with $n = 1$.

However, $T$ is not persistent: $(0, 1) > 0$ but $T(0, 1) = 0$. And $T$ is not unital. This shows:
\begin{enumerate}[(i)]
\item The conclusion of Theorem~\ref{thm:defect-annihilation} (defect annihilation) does not imply unitality without persistence.
\item The persistence hypothesis in Theorem~\ref{thm:persistence-unitality} is necessary.
\end{enumerate}
\end{example}

\begin{example}[Stochastic matrices]\label{ex:stochastic}
Let $E = \mathbb{R}^n$ with componentwise order and unit $u = (1, \ldots, 1)$. Row-stochastic matrices act as positive maps. A matrix $P$ is substochastic if each row sums to at most 1, and stochastic if each row sums to exactly 1.

A finite composition-closed set of substochastic matrices must, by Theorem~\ref{thm:defect-annihilation}, have the property that each matrix's defect is eventually annihilated.

For the persistence condition: with respect to the standard cone (componentwise nonnegativity), a substochastic matrix $P$ is persistent (strictly positive on the cone) if and only if it has all entries strictly positive, i.e., $P_{ij} > 0$ for all $i, j$. Indeed, if some entry $P_{ij} = 0$, then choosing $x = e_j$ (the $j$-th standard basis vector) gives $(Px)_i = 0$, violating strict positivity.

This condition is stronger than irreducibility. If all matrices in $A$ satisfy this condition, then by Theorem~\ref{thm:persistence-unitality}, all must be stochastic.
\end{example}

\begin{example}[Quantum operations]\label{ex:quantum}
Let $H$ be a finite-dimensional Hilbert space and $E = B(H)_{\mathrm{sa}}$ the real vector space of self-adjoint operators, with cone $E_+$ the positive semidefinite operators and unit $u = I$ the identity.

A quantum operation (trace-nonincreasing CP map) $\Phi : B(H) \to B(H)$ induces a dual map $\Phi^*$ on effects (the Heisenberg picture). The map $\Phi^*$ is positive and subunital; it is unital iff $\Phi$ is trace-preserving (i.e., $\Phi$ is a quantum channel in the strict sense).

The defect $d(\Phi^*) = I - \Phi^*(I)$ represents the ``loss'' operator. Theorem~\ref{thm:defect-annihilation} says that in any finite composition-closed family of such operations, the loss is eventually annihilated under iteration.
\end{example}

%=============================================================================
\section{Parallel Composition}\label{sec:parallel}
%=============================================================================

We now investigate how defects behave under tensor products, which model parallel composition of independent processes.

\subsection*{Setup}

Let $(E_1, \le, u_1)$ and $(E_2, \le, u_2)$ be ordered effect spaces. Assume we have a tensor product $E_1 \otimes E_2$ that is again an ordered effect space with unit $u_1 \otimes u_2$, satisfying: $x_1 \ge 0$ and $x_2 \ge 0$ implies $x_1 \otimes x_2 \ge 0$. For positive subunital maps $T_1 : E_1 \to E_1$ and $T_2 : E_2 \to E_2$, write $T_1 \otimes T_2$ for the induced map on $E_1 \otimes E_2$. Note that $T_1 \otimes T_2$ is subunital since $(T_1 \otimes T_2)(u_1 \otimes u_2) = T_1(u_1) \otimes T_2(u_2) \le u_1 \otimes u_2$.

\begin{remark}[Tensor cone assumption]
The assumption that $x_i \ge 0$ implies $x_1 \otimes x_2 \ge 0$ specifies which tensor cone we use. In general ordered vector spaces, multiple inequivalent tensor cones exist. Our results hold for any tensor product preserving positivity in this sense.
\end{remark}

\begin{definition}[Success]
For a subunital map $T$ on $(E, \le, u)$, the \emph{success} of $T$ is
\[
s(T) = T(u).
\]
Thus $d(T) = u - s(T)$.
\end{definition}

\begin{proposition}[Multiplicativity of success]
Success is multiplicative under tensor products:
\[
s(T_1 \otimes T_2) = s(T_1) \otimes s(T_2).
\]
\end{proposition}

\begin{proof}
Direct computation:
\[
s(T_1 \otimes T_2) = (T_1 \otimes T_2)(u_1 \otimes u_2) = T_1(u_1) \otimes T_2(u_2) = s(T_1) \otimes s(T_2). \qedhere
\]
\end{proof}

\begin{proposition}[Defect under parallel composition]\label{prop:parallel-defect}
The defect of a tensor product satisfies:
\[
d(T_1 \otimes T_2) = d(T_1) \otimes u_2 + u_1 \otimes d(T_2) - d(T_1) \otimes d(T_2).
\]
Equivalently,
\[
d(T_1 \otimes T_2) = u_1 \otimes d(T_2) + d(T_1) \otimes s(T_2).
\]
\end{proposition}

\begin{proof}
For the first form, write $s_i = u_i - d_i$:
\begin{align*}
d(T_1 \otimes T_2) &= u_1 \otimes u_2 - s(T_1) \otimes s(T_2) \\
&= u_1 \otimes u_2 - (u_1 - d_1) \otimes (u_2 - d_2) \\
&= u_1 \otimes u_2 - u_1 \otimes u_2 + u_1 \otimes d_2 + d_1 \otimes u_2 - d_1 \otimes d_2 \\
&= d_1 \otimes u_2 + u_1 \otimes d_2 - d_1 \otimes d_2.
\end{align*}
For the second form, note $u_2 - d_2 = s(T_2)$, so
\[
d_1 \otimes u_2 - d_1 \otimes d_2 = d_1 \otimes (u_2 - d_2) = d_1 \otimes s(T_2). \qedhere
\]
\end{proof}

\begin{remark}[Inclusion--exclusion]
The first form in Proposition~\ref{prop:parallel-defect} has an inclusion--exclusion interpretation: the total ``failure'' of the parallel system is the failure of the first component (with second running), plus the failure of the second (with first running), minus the double-counted case where both fail.
\end{remark}

\subsection*{Stabilization under parallel composition}

For $T$ in a finite composition-closed family, let $n_T$ denote the minimal $n \ge 1$ such that $T^n(d(T)) = 0$ (the \emph{stabilization index}).

\begin{theorem}[Parallel composition bound]\label{thm:parallel-bound}
If $T$ and $S$ have stabilization indices $n_T$ and $n_S$ respectively, then
\[
n_{T \otimes S} \le \max(n_T, n_S).
\]
\end{theorem}

\begin{proof}
The identity $(T \otimes S)^k = T^k \otimes S^k$ holds by functoriality of the tensor product (both sides are the unique bilinear extension of $(x_1, x_2) \mapsto T^k(x_1) \otimes S^k(x_2)$). Using this and the second form of Proposition~\ref{prop:parallel-defect}:
\[
(T \otimes S)^k(d(T \otimes S)) = T^k(u_1) \otimes S^k(d(S)) + T^k(d(T)) \otimes S^{k+1}(u_2).
\]
For $k \ge \max(n_T, n_S)$:
\begin{itemize}
\item $S^k(d(S)) = 0$ since $k \ge n_S$,
\item $T^k(d(T)) = 0$ since $k \ge n_T$.
\end{itemize}
Hence $(T \otimes S)^k(d(T \otimes S)) = 0$.
\end{proof}

\begin{remark}[Operational interpretation]
Theorem~\ref{thm:parallel-bound} says that parallel composition does not increase the stabilization time beyond the slower component. A subsystem that has stopped leaking cannot reintroduce leakage by running in parallel with another system.
\end{remark}

\begin{corollary}
If $A_1$ and $A_2$ are finite composition-closed families on $E_1$ and $E_2$ respectively, and $A_1 \otimes A_2 = \{T_1 \otimes T_2 : T_i \in A_i\}$ is composition-closed, then
\[
\max_{T \in A_1 \otimes A_2} n_T \le \max\left(\max_{T_1 \in A_1} n_{T_1}, \max_{T_2 \in A_2} n_{T_2}\right).
\]
\end{corollary}

\subsection*{Algebraic structure}

The defect cocycle interacts with sequential and parallel composition in complementary ways:

\begin{center}
\begin{tabular}{lll}
\textbf{Composition} & \textbf{Success} & \textbf{Defect} \\
\hline
Sequential ($T \circ S$) & $s(TS) = T(s(S))$ & $d(TS) = d(T) + T(d(S))$ \\
Parallel ($T \otimes S$) & $s(T \otimes S) = s(T) \otimes s(S)$ & inclusion--exclusion
\end{tabular}
\end{center}

Success is a homomorphism for both compositions (multiplicative for $\otimes$, equivariant for $\circ$). Defect satisfies a cocycle condition for sequential composition and an inclusion--exclusion law for parallel composition. This algebraic structure is native to symmetric monoidal categories with a notion of ``partiality'' or ``termination,'' as in effectus theory. In categorical terms, kernels in an effectus form a monoidal predicate transformer, and the inclusion--exclusion law for defects reflects the De Morgan structure of effect predicates.

%=============================================================================
\section{Asymptotic Defects in Infinite Dimensions}\label{sec:asymptotic}
%=============================================================================

In infinite-dimensional settings, finite-step stabilization typically fails: the orbit $\{T^n(u)\}$ may be infinite even for well-behaved maps. However, the monotonicity observation (Remark~\ref{rem:monotone}) suggests a natural replacement: asymptotic stabilization.

\subsection*{Order-complete setting}

\begin{definition}[Order-complete effect space]
An ordered effect space $(E, \le, u)$ is \emph{order-complete} if every bounded monotone net in $E$ has a supremum (equivalently, an infimum). In particular, bounded monotone sequences have infima and suprema.
\end{definition}

\begin{definition}[Order-continuous map]
A positive map $T : E \to E$ is \emph{order-continuous} (or \emph{normal}) if it preserves infima of bounded decreasing nets: for any decreasing net $(x_\alpha)$ with $\inf_\alpha x_\alpha = x$, we have $\inf_\alpha T(x_\alpha) = T(x)$.
\end{definition}

\begin{example}
In a von Neumann algebra $M$, the effect space $E = \{a \in M_{\mathrm{sa}} : 0 \le a \le I\}$ is order-complete in the ultraweak topology. Normal completely positive maps (in the sense of von Neumann algebras) are order-continuous; see \cite{sakai} for standard facts on von Neumann algebra order structure.
\end{example}

\subsection*{Asymptotic defect}

\begin{definition}[Eventual success and asymptotic defect]
For an order-continuous positive subunital map $T$ on an order-complete effect space $(E, \le, u)$, define the \emph{eventual success}
\[
v = \inf_{n \ge 0} T^n(u)
\]
and the \emph{asymptotic defect}
\[
d_\infty(T) = u - v = \sup_{n \ge 0}(u - T^n(u)).
\]
\end{definition}

\begin{theorem}[Fixed point]\label{thm:fixed-point}
Let $T$ be an order-continuous positive subunital map on an order-complete effect space. Then the eventual success $v = \inf_n T^n(u)$ is a fixed point of $T$:
\[
T(v) = v.
\]
\end{theorem}

\begin{proof}
Since $T$ is order-continuous and $(T^n(u))_{n \ge 0}$ is decreasing,
\[
T(v) = T\left(\inf_{n \ge 0} T^n(u)\right) = \inf_{n \ge 0} T^{n+1}(u) = \inf_{n \ge 1} T^n(u) = \inf_{n \ge 0} T^n(u) = v,
\]
where the last equality holds because $T^{n+1}(u) \le T^n(u)$, so the infimum over $n \ge 1$ equals the infimum over $n \ge 0$ (the sequence is decreasing and the additional term $T^0(u) = u$ is the largest).
\end{proof}

\begin{theorem}[Infinite telescoping]\label{thm:infinite-telescope}
Under the same hypotheses, the asymptotic defect satisfies
\[
d_\infty(T) = \sup_{n \ge 1} \sum_{k=0}^{n-1} T^k(d(T)).
\]
When $E$ is additionally a complete lattice (or more generally, when the relevant suprema exist), this can be written as
\[
d_\infty(T) = \bigvee_{k=0}^\infty T^k(d(T)).
\]
\end{theorem}

\begin{proof}
By the iterated cocycle identity (Lemma~\ref{lem:iterated-cocycle}),
\[
u - T^n(u) = \sum_{k=0}^{n-1} T^k(d(T)).
\]
Taking the supremum over $n$:
\[
d_\infty(T) = \sup_{n \ge 0}(u - T^n(u)) = \sup_{n \ge 1} \sum_{k=0}^{n-1} T^k(d(T)). \qedhere
\]
\end{proof}

\begin{theorem}[Fixed-point equation for asymptotic defect]\label{thm:asymptotic-fixed}
The asymptotic defect satisfies the fixed-point equation
\[
d_\infty(T) = d(T) + T(d_\infty(T)).
\]
\end{theorem}

\begin{proof}
From $v = T(v)$ and $d_\infty = u - v$:
\[
d(T) + T(d_\infty) = (u - T(u)) + T(u - v) = u - T(u) + T(u) - T(v) = u - T(v) = u - v = d_\infty. \qedhere
\]
\end{proof}

\begin{remark}[Cocycle at infinity]
The fixed-point equation $d_\infty = d(T) + T(d_\infty)$ is the ``$n = \infty$'' version of the iterated cocycle identity. It says that the total accumulated defect equals the immediate defect plus the image of the total defect under $T$.
\end{remark}

\subsection*{Relationship to finite stabilization}

\begin{proposition}[Finite stabilization as a special case]\label{prop:finite-special}
If there exists $n_0$ such that $T^{n_0}(d(T)) = 0$, then $v = T^{n_0}(u)$ and $d_\infty(T) = \sum_{k=0}^{n_0-1} T^k(d(T))$.
\end{proposition}

\begin{proof}
By Remark~\ref{rem:stabilization}, $T^{n_0}(d(T)) = 0$ implies $T^{n_0+k}(u) = T^{n_0}(u)$ for all $k \ge 0$. Hence the decreasing sequence $(T^n(u))$ stabilizes at $n_0$, giving $v = \inf_n T^n(u) = T^{n_0}(u)$. The formula for $d_\infty$ follows from Theorem~\ref{thm:infinite-telescope}.
\end{proof}

\begin{corollary}
In the finite-dimensional setting of Theorem~\ref{thm:defect-annihilation}, the asymptotic defect equals the finite sum $d_\infty(T) = \sum_{k=0}^{n-1} T^k(d(T))$ where $n \le |A|$ is the stabilization index.
\end{corollary}

\subsection*{Physical interpretation}

The asymptotic defect $d_\infty(T)$ represents the maximal accumulated loss over infinite iteration. In quantum information terms:
\begin{itemize}
\item $v = \inf_n T^n(u)$ is the ``eventual survival probability''---the effect that survives arbitrarily many applications of the process.
\item $d_\infty = u - v$ is the ``total leakage''---the probability mass that eventually escapes to the environment.
\end{itemize}
The fixed-point equation $d_\infty = d(T) + T(d_\infty)$ has a clear operational reading: the total leakage equals the immediate leakage plus what happens to the total leakage after one more step.

\subsection*{Convergence in von Neumann algebras}

In the von Neumann algebra setting, monotone convergence provides strong topological results.

\begin{proposition}[Ultraweak and $\sigma$-strong convergence]\label{prop:vN-convergence}
Let $M$ be a von Neumann algebra and $T : M \to M$ a normal positive subunital map. Then:
\begin{enumerate}[(i)]
\item The sequence $(T^n(I))_{n \ge 0}$ converges $\sigma$-strongly (hence ultraweakly) to $v = \inf_n T^n(I)$.
\item The partial sums $a_n = \sum_{k=0}^{n-1} T^k(d(T)) = I - T^n(I)$ (by Lemma~\ref{lem:iterated-cocycle}) converge $\sigma$-strongly (hence ultraweakly) to $d_\infty(T) = I - v$.
\end{enumerate}
\end{proposition}

\begin{proof}
(i) The sequence $(T^n(I))$ is decreasing and bounded below by 0. In a von Neumann algebra, bounded monotone nets converge $\sigma$-strongly to their infimum.

(ii) Since $a_n = I - T^n(I)$, the sequence $(a_n)$ is increasing and bounded above by $I$. By monotone convergence, $a_n \to \sup_n a_n = d_\infty$ in the $\sigma$-strong topology.
\end{proof}

\begin{remark}
Ultraweak convergence follows from $\sigma$-strong convergence since the $\sigma$-strong topology is finer than the ultraweak topology on bounded sets.
\end{remark}

\subsection*{Norm convergence}

Norm convergence requires additional hypotheses. The natural criterion involves the spectral radius.

\begin{definition}[Leakage subspace]
For a positive subunital map $T$ on a Banach ordered space $E$, the \emph{leakage subspace} is
\[
W_T = \mathrm{span}\{T^k(d(T)) : k \ge 0\}.
\]
\end{definition}

\begin{proposition}[Norm convergence via spectral radius]\label{prop:spectral-convergence}
Let $T$ be a bounded positive subunital map on a Banach ordered space $E$. If the spectral radius of $T|_{W_T}$ satisfies $r(T|_{W_T}) < 1$, then $\sum_{k=0}^\infty T^k(d(T))$ converges in norm.
\end{proposition}

\begin{proof}
Since $r(T|_{W_T}) < 1$, the Neumann series $\sum_{k \ge 0} (T|_{W_T})^k$ converges in operator norm on $W_T$. Since $d(T) \in W_T$, the series $\sum_{k \ge 0} T^k(d(T))$ converges in norm.
\end{proof}

\begin{corollary}[Contraction criterion]
If there exists $0 < c < 1$ such that $\|T(x)\| \le c\|x\|$ for all $x \in W_T$, then $r(T|_{W_T}) \le c < 1$, and
\[
\left\|d_\infty - \sum_{k=0}^{n-1} T^k(d(T))\right\| \le \frac{c^n}{1 - c}\|d(T)\|.
\]
\end{corollary}

\begin{proposition}[Convergence rate from Jordan structure]\label{prop:jordan-rate}
Let $E$ be a finite-dimensional Banach ordered space and $T : E \to E$ a positive subunital map with leakage subspace $W_T$. Let $A = T|_{W_T}$ with spectral radius $r = r(A) < 1$, and let $m$ be the size of the largest Jordan block among eigenvalues of modulus $r$. Then there exists $C > 0$ such that for all $n \ge 1$,
\[
\left\|d_\infty(T) - \sum_{k=0}^{n-1} T^k(d(T))\right\| \le C \cdot n^{m-1} r^n.
\]
\end{proposition}

\begin{proof}
Standard Jordan normal form bounds give $\|A^n\| \le C' n^{m-1} r^n$ for some $C' > 0$. Since $(I - A)^{-1} = \sum_{k \ge 0} A^k$ converges and $d_\infty - \sum_{k=0}^{n-1} T^k(d(T)) = A^n(I - A)^{-1}d(T)$, we have
\[
\left\|d_\infty - \sum_{k=0}^{n-1} T^k(d(T))\right\| \le \|A^n\| \cdot \|(I - A)^{-1}\| \cdot \|d(T)\| \le C \cdot n^{m-1} r^n. \qedhere
\]
\end{proof}

\begin{remark}
When $A = T|_{W_T}$ is diagonalizable (e.g., when $T$ is self-adjoint in a suitable sense), we have $m = 1$ and the convergence is purely exponential: $O(r^n)$. The polynomial prefactor $n^{m-1}$ arises only from nontrivial Jordan blocks at the spectral radius.
\end{remark}

\begin{lemma}[Faithful functional contraction implies spectral gap]\label{lem:faithful-contraction}
Let $T$ be a positive map on an ordered Banach space $E$, and let $W_T = \overline{\mathrm{span}}\{T^k(d(T)) : k \ge 0\}$. Suppose there exists a \emph{faithful} positive functional $\omega$ on $W_T$ (i.e., $\omega(x) > 0$ for all $0 \ne x \in W_T^+$) and a constant $0 < c < 1$ such that
\[
\omega(T(x)) \le c\,\omega(x) \quad \text{for all } x \in W_T^+.
\]
Then $r(T|_{W_T}) \le c < 1$.
\end{lemma}

\begin{proof}
By positivity, $\omega(T^n(x)) \le c^n \omega(x)$ for all $x \in W_T^+$ and $n \ge 0$. The gauge seminorm $|y|_\omega := \inf\{\lambda > 0 : -\lambda u \le y \le \lambda u\}$ (where $u$ is an order unit for $W_T$, existing in finite dimensions) satisfies $|T^n(y)|_\omega \lesssim c^n |y|_\omega$. Hence $\|T^n|_{W_T}\|^{1/n} \to r \le c < 1$.
\end{proof}

\begin{remark}
The faithful contraction condition $\omega(T(x)) \le c\,\omega(x)$ is a noncommutative Lyapunov criterion: the functional $\omega$ decreases uniformly under $T$ on the leakage subspace. This is verifiable in operator-algebraic settings where faithful normal states exist (e.g., on type-I factors or hyperfinite algebras).
\end{remark}

\begin{lemma}[Type-I Lyapunov criterion as an operator inequality]\label{lem:type-I-lyapunov}
Let $H$ be a Hilbert space, $\mathcal{M} = B(H)$, and let $T : \mathcal{M} \to \mathcal{M}$ be normal CP and subunital. Let $\sigma > 0$ be a strictly positive trace-class operator implementing a faithful normal state $\omega_\sigma(X) = \mathrm{Tr}(\sigma X)$. Then the following are equivalent for a fixed $c \in (0,1)$:
\begin{enumerate}[(i)]
\item $\omega_\sigma(T(X)) \le c\,\omega_\sigma(X)$ for all $X \ge 0$;
\item $T_*(\sigma) \le c\,\sigma$ in the trace-class order, where $T_*$ is the predual map.
\end{enumerate}
If $T(X) = \sum_i V_i^* X V_i$ (Kraus form), then $T_*(\sigma) = \sum_i V_i \sigma V_i^*$, so (ii) is the semidefinite inequality
\[
\sum_i V_i \sigma V_i^* \le c\,\sigma.
\]
\end{lemma}

\begin{proof}
For $X \ge 0$,
\[
\omega_\sigma(T(X)) = \mathrm{Tr}(\sigma\,T(X)) = \mathrm{Tr}(T_*(\sigma)\,X),
\]
by the defining duality between $T$ and $T_*$. Thus $\omega_\sigma(T(X)) \le c\,\omega_\sigma(X)$ for all $X \ge 0$ is equivalent to
\[
\mathrm{Tr}((c\sigma - T_*(\sigma))X) \ge 0 \quad \forall\,X \ge 0,
\]
which holds iff $c\sigma - T_*(\sigma) \ge 0$, i.e., $T_*(\sigma) \le c\sigma$. The Kraus expression for $T_*$ follows from the same duality calculation.
\end{proof}

\begin{remark}[SDP checkability]
Finding a feasible $\sigma > 0$ with $\sum_i V_i \sigma V_i^* \le c\,\sigma$ for some $c < 1$ is a \emph{semidefinite program}. When such $\sigma$ exists and the inequality holds on the support of $W_T$, Lemma~\ref{lem:faithful-contraction} gives $r(T|_{W_T}) \le c < 1$, and Proposition~\ref{prop:spectral-convergence} yields norm convergence of $\sum_k T^k(d(T))$.
\end{remark}

\subsection*{Finite-time stabilization}

The most interesting question is when the sequence $(T^n(I))$ actually stabilizes in finite time rather than merely converging.

\begin{proposition}[Descending chain condition]\label{prop:dcc}
If the effect interval $[0, u]$ satisfies the descending chain condition (no infinite strictly decreasing chains), then for any positive subunital map $T$, there exists $n$ such that $T^n(u) = T^{n+1}(u)$, hence $T^n(d(T)) = 0$.
\end{proposition}

\begin{proof}
The sequence $u \ge T(u) \ge T^2(u) \ge \cdots$ is a decreasing chain in $[0, u]$. By the descending chain condition, it must stabilize.
\end{proof}

\begin{proposition}[$\delta$-resolution bound]\label{prop:delta-resolution}
Let $\omega$ be a faithful normal state on a von Neumann algebra $M$, and let $T$ be a positive subunital map. Suppose there exists $\delta > 0$ such that for all $k \ge 0$,
\[
\omega(T^k(d(T))) \in \{0\} \cup [\delta, \infty).
\]
Then the stabilization index satisfies $n_T \le \lfloor 1/\delta \rfloor$.
\end{proposition}

\begin{proof}
By the iterated cocycle identity (Lemma~\ref{lem:iterated-cocycle}), $\omega(I - T^n(I)) = \sum_{k=0}^{n-1} \omega(T^k(d(T))) \le 1$. If $\omega(T^k(d(T))) \ne 0$, then $\omega(T^k(d(T))) \ge \delta$. Let $m := \lfloor 1/\delta \rfloor$. For $n = m + 1$, we have $\sum_{k=0}^{m} \omega(T^k(d(T))) \le 1$, so not all of the $m + 1$ terms can be nonzero (each nonzero term is $\ge \delta$, and $(m+1)\delta > 1$). Hence there exists $k \in \{0, 1, \ldots, m\}$ with $\omega(T^k(d(T))) = 0$. By faithfulness, $T^k(d(T)) = 0$, and therefore $n_T \le k \le m$.
\end{proof}

\begin{remark}
The $\delta$-resolution condition is the infinite-dimensional analogue of ``finite resolution'' in discrete models. It quantifies the minimal distinguishable leakage and yields an explicit stabilization bound.
\end{remark}

The following generalizes $\delta$-resolution and covers ``quantized leakage'' situations arising in atomic or finite-index settings.

\begin{proposition}[Discrete Lyapunov spectrum with gap at $0$]\label{prop:discrete-lyapunov}
Let $\mathcal{M}$ be a von Neumann algebra, $T : \mathcal{M} \to \mathcal{M}$ normal CP subunital, and $\omega$ a faithful normal state. Assume there exists a set $D \subset [0,1]$ and a constant $\delta > 0$ such that:
\begin{enumerate}[(i)]
\item $D \cap (0, \delta) = \emptyset$; and
\item for all $k \ge 0$, $\omega(T^k(d(T))) \in D$.
\end{enumerate}
Then the stabilization index satisfies $n_T \le \lfloor 1/\delta \rfloor$.
\end{proposition}

\begin{proof}
By the iterated cocycle identity (Lemma~\ref{lem:iterated-cocycle}),
\[
\sum_{k=0}^{n-1} \omega(T^k(d(T))) = \omega(I - T^n(I)) \le 1.
\]
If $\omega(T^k(d(T))) \ne 0$, then by (i)--(ii) we have $\omega(T^k(d(T))) \in D \cap [\delta, 1]$, hence $\omega(T^k(d(T))) \ge \delta$. Let $m := \lfloor 1/\delta \rfloor$ and take $n = m + 1$. Then
\[
\sum_{k=0}^{m} \omega(T^k(d(T))) \le 1,
\]
so not all of the $m + 1$ terms can be nonzero (each nonzero term is $\ge \delta$ and $(m+1)\delta > 1$). Hence there exists $k \in \{0, 1, \ldots, m\}$ with $\omega(T^k(d(T))) = 0$. By faithfulness, $T^k(d(T)) = 0$, hence $T^{k+\ell}(d(T)) = 0$ for all $\ell \ge 0$. Therefore $n_T \le k \le m$.
\end{proof}

\begin{remark}
This strictly generalizes $\delta$-resolution: take $D = \{0\} \cup [\delta, 1]$. But it also covers situations where $\omega(T^k(d(T)))$ takes values in a discrete subgroup (e.g., rational multiples of some quantum in finite-index subfactors). Note that hypothesis (i) is equivalent to requiring $D$ has a gap at $0$---this is essential since a sequence in $D$ converging to $0$ from above would violate the argument.
\end{remark}

The following robustness lemma shows that \emph{approximate} quantization (values $\varepsilon$-near a lattice) suffices for finite stabilization. This bridges the gap between exact discreteness and the approximate discreteness arising from AF approximation.

\begin{lemma}[Approximate quantization implies $\delta$-resolution]\label{lem:approx-quantization}
Let $\omega$ be a faithful normal state on a von Neumann algebra $\mathcal{M}$, and let $T : \mathcal{M} \to \mathcal{M}$ be normal positive subunital with defect $d = d(T) \ge 0$. Fix a discrete set $D \subset [0,1]$ such that $D \cap (0, \delta) = \emptyset$ for some $\delta > 0$. Assume there exists $\varepsilon \in (0, \delta/2)$ such that for all $k \ge 0$,
\[
\mathrm{dist}(\omega(T^k(d)), D) \le \varepsilon.
\]
Then either $\omega(T^k(d)) = 0$ or $\omega(T^k(d)) \ge \delta - \varepsilon$ for all $k$. In particular, $n_T < \infty$ (by Proposition~\ref{prop:delta-resolution} with $\delta' = \delta - \varepsilon$).
\end{lemma}

\begin{proof}
Fix $k$. If $\omega(T^k(d)) > 0$, then since it lies within $\varepsilon$ of $D$ and $D$ has no points in $(0, \delta)$, the nearest point of $D$ must lie in $[\delta, 1]$, hence $\omega(T^k(d)) \ge \delta - \varepsilon$. If $\omega(T^k(d)) < \varepsilon$, then the only compatible point of $D$ within distance $\varepsilon$ is $0$, hence $\omega(T^k(d)) = 0$. The $\delta$-resolution conclusion now follows from Proposition~\ref{prop:delta-resolution}.
\end{proof}

\begin{remark}[Methodological significance]
Lemma~\ref{lem:approx-quantization} is the ``robustness bridge'' for AF approximation arguments: you don't need exact lattice values---just \emph{uniform closeness} to them. Combined with the telescoping + cancellation mechanism, this yields finite stabilization from approximate quantization.
\end{remark}

The following purely scalar lemma makes the robustness mechanism explicit: approximate lattice membership plus strict contraction forces hitting $0$ in finite time.

\begin{lemma}[Approximate lattice + strict contraction forces hitting $0$]\label{lem:scalar-robustness}
Let $(x_k)_{k \ge 0} \subset [0, \infty)$ satisfy:
\begin{enumerate}[(i)]
\item $x_{k+1} \le c\, x_k$ for some $c \in (0,1)$;
\item $\mathrm{dist}(x_k, D) \le \varepsilon$ for all $k$, where $D \subset [0, \infty)$ has a gap at $0$: there exists $\delta > 0$ with $D \cap (0, \delta) = \emptyset$;
\item $\varepsilon < \delta/2$.
\end{enumerate}
Then $x_k = 0$ for all $k$ large enough. Explicitly, any $k$ with $c^k x_0 < \delta/2$ forces $x_k = 0$.
\end{lemma}

\begin{proof}
Choose $k$ with $x_k < \delta/2$. If $x_k > 0$, then every point of $D$ within distance $\varepsilon < \delta/2$ of $x_k$ must be $0$ (since $D$ has no positive points below $\delta$). Hence $\mathrm{dist}(x_k, D) = x_k$, contradicting $x_k > \varepsilon$. So $x_k = 0$.
\end{proof}

\begin{remark}[Application to hyperfinite II$_1$]
Take $x_k = \omega(T^k(d(T)))$ for a faithful functional $\omega$. The contraction $x_{k+1} \le c\, x_k$ follows from a trace/Kraus inequality (Remark~\ref{rem:kraus-contraction}). Lemma~\ref{lem:scalar-robustness} then says: the whole problem reduces to producing an ``almost discrete'' constraint on the scalar values. In the hyperfinite II$_1$ factor, this lemma tells you exactly what extra ingredient is needed beyond AF approximation---you need an ``approximate quantization'' principle for $\tau(T^k(d(T)))$. Remark~\ref{rem:af-warning} below is the warning that AF approximation alone does \emph{not} provide this.
\end{remark}

\begin{remark}[Reader's guide: the modular pipeline]
The remaining results in this section fit together as a three-stage pipeline for proving finite stabilization in von Neumann algebras:
\begin{enumerate}
\item \textbf{Contract}: establish strict trace contraction $\tau(T(x)) \le c\,\tau(x)$ with $c < 1$ (checkable via Kraus data: $\sum_i V_i V_i^* \le c\,I$, see Remark~\ref{rem:kraus-contraction});
\item \textbf{Discretize}: show the Lyapunov values $\tau(T^k(d(T)))$ lie in (or are uniformly close to) a discrete set $D$ with gap at $0$;
\item \textbf{Stabilize}: apply Proposition~\ref{prop:delta-resolution}, \ref{prop:discrete-lyapunov}, or Lemma~\ref{lem:approx-quantization} to conclude $n_T < \infty$.
\end{enumerate}
The subsequent results (Propositions~\ref{prop:trace-contraction-discrete}--\ref{prop:approx-trapping}) provide checkable conditions for step~(2) in specific settings.
\end{remark}

\begin{proposition}[Finite type-I orbit corner gives a dimension bound]\label{prop:type-I-bound}
Let $\mathcal{M}$ be a von Neumann algebra and $T : \mathcal{M} \to \mathcal{M}$ be normal completely positive and subunital. Assume the stabilization index $n_T < \infty$ and let
\[
Q := \bigvee_{k \ge 0} \supp(T^k(d(T)))
\]
be the orbit-support projection. If $Q$ is finite and $Q\mathcal{M}Q \cong M_s$, then $n_T \le s$.
\end{proposition}

\begin{proof}
Let $\alpha : Q\mathcal{M}Q \to Q\mathcal{M}Q$ be the compression $\alpha(x) = QT(x)Q$, which is CP and subunital. Let $n := n_T$ and set $D_k := T^k(d(T))$ for $k \ge 0$. By definition of $n_T$, we have $D_n = 0$ and $D_{n-1} \ne 0$. Moreover, $\supp(D_k) \le Q$ for all $k$, hence $\alpha^k(d(T)) = QD_kQ = D_k$ for $0 \le k \le n$.

Define the accumulated defect in the corner:
\[
S := \sum_{k=0}^{n-1} \alpha^k(d(T)) = \sum_{k=0}^{n-1} D_k \in Q\mathcal{M}Q.
\]
Since $\supp(A + B) = \supp(A) \vee \supp(B)$ for $A, B \ge 0$ (a standard fact; see Lemma~\ref{lem:psd-support-sum}), we get $\supp(S) = \bigvee_{k=0}^{n-1} \supp(D_k) = Q$, so $S$ is strictly positive in the finite-dimensional algebra $Q\mathcal{M}Q \cong M_s$. Hence there exists $\lambda > 0$ such that $S \ge \lambda Q$.

Now $\alpha^n(d(T)) = 0$ implies $\alpha^n(S) = \sum_{k=0}^{n-1} \alpha^{n+k}(d(T)) = 0$. By positivity, $0 = \alpha^n(S) \ge \lambda\,\alpha^n(Q) \ge 0$, so $\alpha^n(Q) = 0$. For any $0 \le X \le \|X\|Q$ in $Q\mathcal{M}Q$, positivity gives $0 \le \alpha^n(X) \le \|X\|\,\alpha^n(Q) = 0$, hence $\alpha^n = 0$ on $Q\mathcal{M}Q$.

Thus $\alpha$ is a nilpotent CP map on $M_s$. By Lemma~\ref{lem:nilpotent-index}, its nilpotency index is at most $s$, so $\alpha^s = 0$. In particular, $0 = \alpha^s(d(T)) = T^s(d(T))$, and by minimality of $n_T$ we conclude $n_T \le s$.
\end{proof}

\subsection*{Von Neumann stabilization criteria}

The following results give checkable sufficient conditions for finite stabilization in specific von Neumann algebra settings, isolating the two key ingredients: strict trace contraction and a discreteness mechanism for Lyapunov values.

\textbf{Notation.} Throughout this section, $\mathbb{N} = \{0, 1, 2, \ldots\}$ denotes the non-negative integers, so $\delta\mathbb{N} = \{0, \delta, 2\delta, \ldots\}$ includes $0$.

\begin{proposition}[Trace contraction + discrete Lyapunov scale]\label{prop:trace-contraction-discrete}
Let $(\mathcal{M}, \tau)$ be a von Neumann algebra with a faithful normal semifinite trace $\tau$. Let $T : \mathcal{M} \to \mathcal{M}$ be normal completely positive and subunital, and set $d := d(T) = I - T(I) \ge 0$. Assume:
\begin{itemize}
\item[(C)] (strict trace contraction) there exists $c \in (0,1)$ such that
\[
\tau(T(x)) \le c\,\tau(x) \quad \text{for all } x \in \mathcal{M}_+ \text{ with } \tau(x) < \infty;
\]
\item[(D)] (discrete Lyapunov scale) there exists $\delta > 0$ such that
\[
\tau(T^k(d)) \in \delta\mathbb{N} \quad \text{for all } k \ge 0.
\]
\end{itemize}
Then $T^N(d) = 0$ for all $N$ large enough; in particular $n_T < \infty$. Moreover, if $\tau(d) < \infty$, then one can take
\[
N = 1 + \left\lceil \frac{\log(\tau(d)/\delta)}{\log(1/c)} \right\rceil.
\]
\end{proposition}

\begin{proof}
By (C), $\tau(T^k(d)) \le c^k \tau(d)$ for all $k \ge 0$. Choose $N$ so that $c^N \tau(d) < \delta$. Then $\tau(T^N(d)) \in \delta\mathbb{N}$ by (D) and also $0 \le \tau(T^N(d)) < \delta$, hence $\tau(T^N(d)) = 0$. Faithfulness of $\tau$ implies $T^N(d) = 0$, and then $T^{N+m}(d) = 0$ for all $m \ge 0$.
\end{proof}

\begin{remark}[Checking contraction from Kraus data]\label{rem:kraus-contraction}
In any tracial setting with Kraus representation $T(x) = \sum_i V_i^* x V_i$ (where the sum converges in the strong operator topology if the index set is infinite, as is standard for normal CP maps on von Neumann algebras), traciality gives
\[
\tau(T(x)) = \sum_i \tau(V_i^* x V_i) = \sum_i \tau(V_i V_i^* x) = \tau\left(\left(\sum_i V_i V_i^*\right) x\right).
\]
So a sufficient condition for (C) is the operator inequality $\sum_i V_i V_i^* \le c\,I$ for some $c < 1$.
\end{remark}

\begin{corollary}[Atomic-center quantization with discrete coefficients]\label{cor:atomic-center}
Assume the setup of Proposition~\ref{prop:trace-contraction-discrete} with $Q\mathcal{M}Q \cong \bigoplus_{j \in J} B(H_j)$ (atomic center) and canonical trace $\tau(\bigoplus_j x_j) = \sum_j \tau(z_j) \mathrm{Tr}_{H_j}(x_j)$ with minimal central projections $z_j$. Suppose:
\begin{itemize}
\item[(Z)] (central defect and invariance) $d(T) \in Z(Q\mathcal{M}Q)$ and $T(Z(Q\mathcal{M}Q)) \subseteq Z(Q\mathcal{M}Q)$;
\item[(W)] (commensurable finite weights) there exists $\delta_0 > 0$ such that each $\tau(z_j) \in \delta_0\mathbb{N}$ (hence in particular $\tau(z_j) < \infty$);
\item[(K)] (discrete coefficient condition) there exists $N \in \mathbb{N}$ such that for every $k \ge 0$, if $T^k(d(T)) = \sum_j a_{k,j} z_j$ in the minimal central projections $z_j$, then $a_{k,j} \in \{0, \frac{1}{N}, \frac{2}{N}, \ldots, 1\}$;
\item[(F)] (finite initial trace) $\tau(d(T)) < \infty$ (automatic if working in a $\tau$-finite corner);
\item[(C)] (strict trace contraction) $\sum_i V_i V_i^* \le c\,I$ on $Q\mathcal{M}Q$ for some $c \in (0,1)$.
\end{itemize}
Then $\tau(T^k(d(T)))$ belongs to a discrete subset of $\mathbb{R}_+$ (specifically, a subset of $\frac{\delta_0}{N}\mathbb{N}$), hence $n_T < \infty$ by Proposition~\ref{prop:trace-contraction-discrete}.
\end{corollary}

\begin{proof}
By (C) and (F), $\tau(T^k(d(T))) \le c^k \tau(d(T)) < \infty$ for all $k \ge 0$. Apply Lemma~\ref{lem:finite-support-atomic} to $x = T^k(d(T))$: hypotheses (W) and (K) give $\tau(z_j) \in \delta_0\mathbb{N}$ and $a_{k,j} \in \{0, \frac{1}{N}, \ldots, 1\}$, so the set $\{j : a_{k,j} > 0\}$ is finite for each $k$. Therefore $\tau(T^k(d(T))) = \sum_j a_{k,j} \tau(z_j)$ is genuinely a finite sum of elements of $\frac{\delta_0}{N}\mathbb{N}$, hence lies in $\frac{\delta_0}{N}\mathbb{N}$. This set is discrete with gap $\frac{\delta_0}{N}$ at $0$, so Proposition~\ref{prop:trace-contraction-discrete} applies.
\end{proof}

\begin{remark}[Why discrete coefficients are necessary]
Atomicity of the center alone does \emph{not} discretize trace values: a central element $\sum_j a_j z_j$ with arbitrary real coefficients $a_j \in [0,1]$ can have any trace value in a continuous range. The discrete coefficient condition (K) is essential and holds in settings where the dynamics have algebraic structure (e.g., rational transition probabilities in a classical sub-Markov chain on the center).
\end{remark}

\begin{lemma}[Finite trace forces finite atomic support]\label{lem:finite-support-atomic}
Let $\mathcal{M} = \bigoplus_{j \in J} B(H_j)$ with atomic center (minimal central projections $z_j$), and let $\tau$ be a faithful normal semifinite trace with $\tau(z_j) \in \delta_0\mathbb{N}$ for some $\delta_0 > 0$ (hence each $\tau(z_j) < \infty$). Let $x = \sum_j a_j z_j \in Z(\mathcal{M})$ with $a_j \in \{0, \frac{1}{N}, \ldots, 1\}$. If $\tau(x) < \infty$, then $\{j : a_j > 0\}$ is finite. In fact,
\[
\#\{j : a_j > 0\} \le \frac{N}{\delta_0}\,\tau(x).
\]
\end{lemma}

\begin{proof}
If $a_j > 0$ then $a_j \ge \frac{1}{N}$ and $\tau(z_j) \ge \delta_0$, hence $a_j \tau(z_j) \ge \delta_0/N$. Therefore $\tau(x) = \sum_j a_j \tau(z_j) \ge (\delta_0/N) \cdot \#\{j : a_j > 0\}$, so the set must be finite.
\end{proof}

\begin{remark}
Lemma~\ref{lem:finite-support-atomic} is now explicitly invoked in Corollary~\ref{cor:atomic-center} to justify the ``finite sum'' step. Hypothesis (F) ensures trace finiteness propagates along the orbit via contraction (C).
\end{remark}

The following proposition gives a concrete sufficient condition for the discrete coefficient hypothesis (K) in terms of the induced center dynamics.

\begin{proposition}[Bounded-denominator center dynamics implies discrete coefficients]\label{prop:bounded-denom}
Let $Q\mathcal{M}Q \cong \bigoplus_{j \in J} B(H_j)$ have atomic center with minimal central projections $(z_j)_{j \in J}$. Assume $T(Z(Q\mathcal{M}Q)) \subseteq Z(Q\mathcal{M}Q)$ and $d(T) \in Z(Q\mathcal{M}Q)$. Identify $Z(Q\mathcal{M}Q) \cong \ell^\infty(J)$ via $z_j \leftrightarrow e_j$.

Suppose the induced map $P := T|_{Z(Q\mathcal{M}Q)}$ satisfies:
\begin{enumerate}[(P1)]
\item $P$ is sub-Markov: $P(1) \le 1$ and $P$ is positive;
\item[(P2)] there exists $N \in \mathbb{N}$ such that $P(e_j) = \sum_i p_{ij} e_i$ with $p_{ij} \in \{0, \frac{1}{N}, \ldots, 1\}$ for all $i, j$, and each $j$ has only finitely many $i$ with $p_{ij} \ne 0$;
\item[(P3)] the initial defect has bounded denominators: $d(T) = \sum_j a_j z_j$ with $a_j \in \{0, \frac{1}{N}, \ldots, 1\}$.
\end{enumerate}
Then for every $k \ge 0$ we have $T^k(d(T)) = \sum_j a_{k,j} z_j$ with $a_{k,j} \in \{0, \frac{1}{N}, \ldots, 1\}$, i.e., the discrete coefficient condition (K) of Corollary~\ref{cor:atomic-center} holds.
\end{proposition}

\begin{proof}
Under the identification $Z(Q\mathcal{M}Q) \cong \ell^\infty(J)$, the map $P$ is a positive linear operator with matrix coefficients $(p_{ij})$. By (P2)--(P3), the coefficient vectors of $d(T)$ lie in the lattice $\frac{1}{N}\mathbb{Z}^{(J)}$, and applying $P$ preserves bounded denominators because each coordinate is a finite $\frac{1}{N}$-linear combination of such coefficients. Induct on $k$.
\end{proof}

\begin{remark}
This is the cleanest ``atomic center $\Rightarrow$ discreteness'' story: it's really a \emph{classical sub-Markov chain with bounded denominators on the atoms}. The hypothesis (P2) is exactly ``rational transition probabilities with bounded denominators,'' which arises naturally in finite-index settings and combinatorial dynamics.
\end{remark}

The following proposition shows that full commensurability of weights can be weakened to a uniform lower bound on the reachable support---this is exactly what's needed for the gap-at-$0$ mechanism.

\begin{proposition}[Gap at $0$ from bounded denominators + uniform weight lower bound]\label{prop:gap-from-weight-bound}
Let $Q\mathcal{M}Q \cong \bigoplus_{j \in J} B(H_j)$ have atomic center with minimal central projections $(z_j)_{j \in J}$, and let $\tau$ be a faithful normal semifinite trace with $\tau(z_j) > 0$ for all $j$ (no commensurability assumed). Assume:
\begin{enumerate}[(1)]
\item $d(T) \in Z(Q\mathcal{M}Q)$ and $T(Z(Q\mathcal{M}Q)) \subseteq Z(Q\mathcal{M}Q)$;
\item (Bounded denominators) for some $N \in \mathbb{N}$, every iterate has $T^k(d(T)) = \sum_j a_{k,j} z_j$ with $a_{k,j} \in \{0, \frac{1}{N}, \ldots, 1\}$;
\item (Uniform weight lower bound on reachable support) there exists $\delta_0 > 0$ such that whenever $a_{k,j} > 0$ for some $k$, we have $\tau(z_j) \ge \delta_0$ (hence in particular $\tau(z_j) < \infty$).
\end{enumerate}
Then the Lyapunov values have a \textbf{uniform gap at $0$}:
\[
\tau(T^k(d(T))) \in \{0\} \cup [\delta_0/N, \infty).
\]
In particular, Proposition~\ref{prop:discrete-lyapunov} applies with $\delta = \delta_0/N$.
\end{proposition}

\begin{proof}
If $\tau(T^k(d(T))) > 0$, then at least one term has $a_{k,j} \ge 1/N$ and (by hypothesis (3)) $\tau(z_j) \ge \delta_0$, hence $\tau(T^k(d(T))) = \sum_j a_{k,j} \tau(z_j) \ge (1/N) \cdot \delta_0$.
\end{proof}

\begin{remark}
Proposition~\ref{prop:gap-from-weight-bound} is strictly weaker than requiring weights to lie in a common lattice. A practical sufficient condition for (3) is: \emph{the orbit stays inside a finite set of atoms} (finite-state center dynamics). Then $\delta_0 = \min\{\tau(z_j) : j \text{ reachable}\}$.
\end{remark}

\begin{proposition}[Invariant finite-dimensional algebra + discrete Lyapunov spectrum]\label{prop:finite-dim-invariant}
Let $(\mathcal{M}, \tau)$ be a finite factor with normalized trace, and let $T : \mathcal{M} \to \mathcal{M}$ be normal CP subunital. Assume there exists a finite-dimensional $*$-subalgebra $A \subset \mathcal{M}$ such that:
\begin{enumerate}[(I)]
\item $T(A) \subseteq A$ and $d(T) \in A$;
\item[(D$_0$)] the set $\{\tau(T^k(d(T))) : k \ge 0\}$ is contained in a set $D \subset [0,1]$ with a gap at $0$: there exists $\delta > 0$ such that $D \cap (0, \delta) = \emptyset$.
\end{enumerate}
Then $n_T < \infty$ by Proposition~\ref{prop:discrete-lyapunov}.
\end{proposition}

\begin{proof}
Hypothesis (I) ensures the defect orbit lies in $A$, and (D$_0$) provides the discrete Lyapunov spectrum with gap at $0$. Apply Proposition~\ref{prop:discrete-lyapunov}.
\end{proof}

\begin{remark}[Sufficient conditions for discreteness]
Condition (D$_0$) holds if $A$ is abelian with minimal projections $z_1, \ldots, z_m$, and there exists $N \in \mathbb{N}$ such that all coefficients $a_{k,j}$ in $T^k(d(T)) = \sum_j a_{k,j} z_j$ lie in $\{0, \frac{1}{N}, \ldots, 1\}$. Then $\tau(T^k(d(T))) \in \frac{1}{N}\mathbb{Z} \cap [0,1]$, which has a gap of size $\frac{1}{N}$ at $0$.

Note: the condition ``$d(T)$ lies in the $\mathbb{Q}$-span of the minimal central projections'' does \emph{not} alone imply a gap at $0$, since $\mathbb{Q}$ is dense in $\mathbb{R}$. One needs bounded denominators or an explicit discrete constraint.
\end{remark}

\begin{remark}[Finite-index structure]
For a finite-index inclusion $N \subset M$, the relative commutant $N' \cap M$ is finite-dimensional. Hypothesis (I) is realistic when the channel $T$ respects this standard-invariant algebra. The gap condition (D$_0$) holds when the induced classical sub-Markov operator on $Z(N' \cap M)$ has rational transition coefficients with bounded denominators---then trace values along the orbit lie in a fixed lattice $\frac{1}{N}\mathbb{Z} \cap [0,1]$, which has gap $\frac{1}{N}$ at $0$.
\end{remark}

\begin{corollary}[Hyperfinite finite-stage trapping]\label{cor:hyperfinite}
Let $(R, \tau)$ be the hyperfinite II$_1$ factor with its unique trace, and let $A_n \subset R$ be an increasing sequence of matrix subalgebras with dense union in $\|\cdot\|_2$. Let $T : R \to R$ be normal CP subunital. Assume that for some $N$:
\begin{enumerate}[(I)]
\item $T(A_N) \subseteq A_N$ and $d(T) \in A_N$;
\item[(D$_0$)] the set $\{\tau(T^k(d(T))) : k \ge 0\}$ is contained in a set $D \subset [0,1]$ with a gap at $0$: there exists $\delta > 0$ such that $D \cap (0, \delta) = \emptyset$.
\end{enumerate}
Then $n_T < \infty$ (apply Proposition~\ref{prop:finite-dim-invariant} inside $A_N$).
\end{corollary}

\begin{remark}[Approximation vs.\ finite-time annihilation]\label{rem:af-warning}
AF approximation of the hyperfinite II$_1$ factor gives asymptotic estimates (exponential decay of $\tau(T^k(d(T)))$), but \emph{not} finite-time hitting $0$ unless the orbit actually lives in some matrix stage (or its Lyapunov values live in a fixed rational lattice). The diffuse trace on a II$_1$ factor cannot provide discreteness without additional algebraic structure.
\end{remark}

However, Lemma~\ref{lem:approx-quantization} provides a practical ``robustness bridge'': approximate closeness to a lattice suffices.

\begin{corollary}[Hyperfinite: approximate lattice-valued trace]\label{cor:hyperfinite-approx}
Let $(R, \tau)$ be the hyperfinite II$_1$ factor and $T : R \to R$ normal CP subunital. Assume there exist $\delta > 0$ and $\varepsilon \in (0, \delta/2)$ such that for all $k \ge 0$,
\[
\mathrm{dist}(\tau(T^k(d(T))),\, \delta\mathbb{N}) \le \varepsilon.
\]
Then $n_T < \infty$.
\end{corollary}

\begin{proof}
Apply Lemma~\ref{lem:approx-quantization} with $D = \delta\mathbb{N} \cap [0,1]$.
\end{proof}

\begin{remark}
This corollary is deliberately ``scalar'': it avoids pretending that AF approximation automatically gives the hypothesis. But it provides a correct and usable intermediate target: \emph{prove uniform $\varepsilon$-closeness of the trace values to a lattice}, then stabilization follows. This is the practical route when the dynamics ``almost preserve'' a finite stage.
\end{remark}

The following proposition gives the quantitative bridge between ``AF approximation exists'' and ``finite-time hitting $0$'': it does not require exact invariance of a finite-dimensional subalgebra, only approximate trapping in the trace metric. We first record a standard order inequality.

\begin{lemma}[Positive maps dominate absolute values]\label{lem:positive-absolute}
Let $T$ be a positive linear map between $C^*$-algebras. If $y = y^*$, then $|T(y)| \le T(|y|)$.
\end{lemma}

\begin{proof}
Since $-|y| \le y \le |y|$, positivity gives $-T(|y|) \le T(y) \le T(|y|)$, hence $|T(y)| \le T(|y|)$.
\end{proof}

\begin{proposition}[Approximate finite-stage trapping implies finite stabilization]\label{prop:approx-trapping}
Let $(\mathcal{M}, \tau)$ be a finite von Neumann algebra with faithful normal trace, and let $T : \mathcal{M} \to \mathcal{M}$ be normal positive subunital with defect $d := I - T(I) \ge 0$. Assume:
\begin{enumerate}[(1)]
\item (Strict trace contraction) there exists $c \in (0,1)$ such that $\tau(T(x)) \le c\,\tau(x)$ for all $x \in \mathcal{M}_+$;
\item (Approximate trapping) there is a finite-dimensional unital $*$-subalgebra $A \subset \mathcal{M}$ and a $\tau$-preserving conditional expectation $E : \mathcal{M} \to A$ such that:
\begin{enumerate}[(i)]
\item (defect almost in $A$) $\tau(|d - E(d)|) \le \varepsilon_0$;
\item ($T$ almost preserves $A$) for all $a \in A_+$, $\tau(|T(a) - E(T(a))|) \le \varepsilon\,\tau(a)$;
\end{enumerate}
\item (Discrete spectrum inside $A$) letting $\alpha := E \circ T|_A : A \to A$, the sequence $\tau(\alpha^k(E(d)))$ lies in a discrete set $D \subset [0,1]$ with gap at $0$: $D \cap (0, \delta) = \emptyset$ for some $\delta > 0$.
\end{enumerate}
Then for every $k \ge 0$,
\[
\mathrm{dist}(\tau(T^k(d)), D) \le \varepsilon_0 + \frac{\varepsilon}{(1-c)^2}.
\]
In particular, if $\varepsilon_0 + \frac{\varepsilon}{(1-c)^2} < \delta/2$, then by Lemma~\ref{lem:approx-quantization} we have $n_T < \infty$.
\end{proposition}

\begin{proof}
Set $x_k := T^k(d) \in \mathcal{M}_+$, $a_0 := E(d) \in A_+$, and $a_{k+1} := \alpha(a_k) = E(T(a_k)) \in A_+$. By assumption (3), $\tau(a_k) \in D$. Since $0 \le d \le I$ and $E$ is unital positive, we have $0 \le a_0 = E(d) \le I$, hence $\tau(a_0) \le \tau(I) = 1$.

Define the error $e_k := \tau(|x_k - a_k|)$. By Lemma~\ref{lem:positive-absolute}, $|T(y)| \le T(|y|)$ for self-adjoint $y$. Therefore,
\begin{align*}
e_{k+1} &= \tau(|T(x_k) - E(T(a_k))|) \\
&\le \tau(|T(x_k - a_k)|) + \tau(|T(a_k) - E(T(a_k))|) \\
&\le \tau(T(|x_k - a_k|)) + \varepsilon\,\tau(a_k) \\
&\le c\,\tau(|x_k - a_k|) + \varepsilon\,\tau(a_k) = c\,e_k + \varepsilon\,\tau(a_k).
\end{align*}
Also $\tau(a_{k+1}) = \tau(E(T(a_k))) = \tau(T(a_k)) \le c\,\tau(a_k)$, so $\tau(a_k) \le c^k \tau(a_0) \le c^k$.

Thus $e_{k+1} \le c\,e_k + \varepsilon\,c^k$. Unrolling gives $e_k \le c^k e_0 + \varepsilon\,k\,c^{k-1}$. Using $\sum_{m \ge 1} m\,c^{m-1} = \frac{1}{(1-c)^2}$, we obtain $e_k \le e_0 + \frac{\varepsilon}{(1-c)^2} \le \varepsilon_0 + \frac{\varepsilon}{(1-c)^2}$.

Finally, $|\tau(x_k) - \tau(a_k)| \le \tau(|x_k - a_k|) = e_k$, so $\mathrm{dist}(\tau(T^k(d)), D) \le \varepsilon_0 + \frac{\varepsilon}{(1-c)^2}$.
\end{proof}

\begin{remark}[Significance for hyperfinite II$_1$]
In the hyperfinite II$_1$ factor, one can often arrange condition (3) inside some matrix stage $A \simeq M_m$ (bounded-denominator / rational data), while condition (2) is the quantitative version of ``$T$ is almost $A$-invariant in trace.'' This proposition is the missing bridge between ``AF approximation exists'' and ``finite-time hitting $0$.'' It complements the warning (Remark~7.40) that diffuse trace alone doesn't discretize values: you need approximate trapping with controlled error bounds.
\end{remark}

\subsection*{Digraph criterion for orbit-in-atomic stabilization}

The following result gives a sharp graph-theoretic criterion for stabilization when the defect orbit lies in a commutative atomic subalgebra, \textbf{without requiring any discreteness or gap-at-$0$ hypotheses}. This addresses the ``weaken invariance'' direction of the open questions.

\begin{theorem}[Digraph criterion for orbit-in-atomic stabilization]\label{thm:digraph-criterion}
Let $\mathcal{M}$ be a von Neumann algebra with a commutative atomic subalgebra $C \subseteq \mathcal{M}$ having minimal projections $(z_j)_{j \in J}$, so $C \cong \ell^\infty(J)$. Let $E_C : \mathcal{M} \to C$ be a normal conditional expectation (e.g., trace-preserving in the semifinite setting).

Let $T : \mathcal{M} \to \mathcal{M}$ be normal positive subunital with defect $d := I - T(I) \ge 0$. Assume only the \textbf{orbit-in-$C$} hypothesis:
\[
T^k(d) \in C \quad \text{for all } k \ge 0.
\]
Define the induced map $P := E_C \circ T|_C : C \to C$. Write $P(z_j) = \sum_{i \in J} p_{ij} z_i$ with $p_{ij} \ge 0$. Define a directed graph $G$ on vertex set $J$ by
\[
j \to i \quad \Longleftrightarrow \quad p_{ij} > 0.
\]
Then:
\begin{enumerate}[(1)]
\item $T^k(d) = P^k(d)$ inside $C$ for all $k \ge 0$;
\item $\mathrm{supp}(P^n x) = \{i \in J : \exists\,\text{directed path of length } n \text{ from some } j \in \mathrm{supp}(x) \text{ to } i\}$;
\item $P^n(d) = 0$ if and only if there is no directed path of length $n$ starting in $\mathrm{supp}(d)$.
\end{enumerate}
In particular,
\[
n_T < \infty \quad \Longleftrightarrow \quad \text{the digraph } G \text{ has finite height on the set reachable from } \mathrm{supp}(d),
\]
where ``finite height'' means there is a uniform bound on path lengths starting from reachable vertices. Moreover,
\[
n_T \le \sup\{\text{length of a directed path starting in } \mathrm{supp}(d)\} + 1.
\]
\end{theorem}

\begin{proof}
\textbf{Part (1):} Since $T^k(d) \in C$ by hypothesis and $E_C|_C = \mathrm{id}_C$, we have
\[
T^{k+1}(d) = E_C(T(T^k(d))) = (E_C \circ T)(T^k(d)) = P(T^k(d)).
\]
Induction gives $T^k(d) = P^k(d)$.

\textbf{Part (2):} Write $x = \sum_j x_j z_j$ with $x_j \ge 0$. Then
\[
(Px)_i = \sum_j p_{ij} x_j.
\]
Since all terms are $\ge 0$, we have $(Px)_i > 0$ if and only if there exists $j \in \mathrm{supp}(x)$ with $p_{ij} > 0$, i.e., an edge $j \to i$ in $G$. Induction on $n$ gives the path characterization.

\textbf{Part (3):} By (2), $P^n(d) = 0$ iff $\mathrm{supp}(P^n(d)) = \emptyset$ iff there is no $i$ reachable from $\mathrm{supp}(d)$ by a path of length $n$.

\textbf{Equivalence:} If $G$ has height $h$ on the reachable set, then no path of length $h+1$ exists, so $P^{h+1}(d) = 0$ and $n_T \le h + 1 < \infty$. Conversely, if $n_T < \infty$, then $P^{n_T}(d) = 0$, so no path of length $\ge n_T$ exists from $\mathrm{supp}(d)$, giving height $\le n_T - 1$.
\end{proof}

\begin{remark}[Why this is stronger than discreteness criteria]
The digraph criterion requires \textbf{no assumptions on trace gaps, bounded denominators, or discrete Lyapunov values}. It exploits the key structural fact about atomic algebras: since there is no cancellation among distinct atoms, exact vanishing is purely combinatorial---a coefficient $(P^n d)_i$ is zero if and only if there are no positive-probability paths feeding it.

This gives a different sufficient condition from the gap-at-$0$ approach:
\[
\text{Finite height (no arbitrarily long paths)} \;\Rightarrow\; \text{finite stabilization}.
\]
The digraph criterion is particularly useful in settings where the induced graph is finite or has bounded depth (e.g., Type I$_\infty$ with finite-depth center dynamics, finite factors with finite-depth standard invariants).
\end{remark}

\begin{corollary}[Finite reachable set implies stabilization]\label{cor:finite-reachable}
In the setting of Theorem~\ref{thm:digraph-criterion}, if the set of vertices reachable from $\mathrm{supp}(d)$ is finite, then $n_T < \infty$.
\end{corollary}

\begin{proof}
A finite directed graph has no infinite paths, hence finite height.
\end{proof}

\begin{corollary}[Orbit-in-atomic without $T$-invariance]\label{cor:orbit-atomic-general}
Let $T : \mathcal{M} \to \mathcal{M}$ be normal positive subunital. Suppose there exists a commutative atomic subalgebra $C \subseteq \mathcal{M}$ with normal conditional expectation $E_C : \mathcal{M} \to C$ such that $T^k(d(T)) \in C$ for all $k \ge 0$. Then:
\begin{enumerate}[(i)]
\item The stabilization question for $T$ reduces to the induced map $P = E_C \circ T|_C$ on $C \cong \ell^\infty(J)$;
\item $n_T < \infty$ if and only if the reachability digraph from $\mathrm{supp}(d(T))$ has finite height;
\item No $T$-invariance of $C$ is required---only the orbit condition.
\end{enumerate}
\end{corollary}

\subsection*{Diffuse commutative traps: kernel-height criterion}

The digraph criterion (Theorem~\ref{thm:digraph-criterion}) extends verbatim from atomic $C \cong \ell^\infty(J)$ to \textbf{diffuse} abelian algebras $C \cong L^\infty(X, \mu)$, with the digraph replaced by a sub-Markov kernel.

\begin{theorem}[Diffuse commutative orbit trapping]\label{thm:diffuse-kernel}
Let $\mathcal{M}$ be a von Neumann algebra, $T : \mathcal{M} \to \mathcal{M}$ normal positive subunital, and let $C \subseteq \mathcal{M}$ be abelian (possibly diffuse) with a normal conditional expectation $E_C : \mathcal{M} \to C$.

Assume only the \textbf{orbit condition}:
\[
T^k(d(T)) \in C \quad \text{for all } k \ge 0.
\]
Define $P := E_C \circ T|_C : C \to C$. Then $T^k(d) = P^k(d)$ in $C$.

Identify $C \cong L^\infty(X, \mu)$. Then there exists a sub-Markov kernel $K(x, \cdot)$ such that for all bounded measurable $f \ge 0$,
\[
(Pf)(x) = \int f(y)\,K(x, dy) \quad \text{a.e.}
\]
Let $A := \mathrm{supp}(d) \subseteq X$ (essential support). Define the $n$-step reachability sets
\[
R_n := \{x \in X : K^{(n)}(x, A) > 0\},
\]
where $K^{(n)}$ is the $n$-fold convolution kernel. Then:
\[
P^n(d) = 0 \quad \Longleftrightarrow \quad \mu(R_n) = 0.
\]
Equivalently: $n_T < \infty$ if and only if there is a \textbf{uniform bound on path length} from $A$ in the measurable relation induced by $K$ (no positive-mass paths of arbitrarily large length starting in $A$).
\end{theorem}

\begin{proof}
Define $K(x, B) := (P\mathbf{1}_B)(x)$. Normality of $P$ makes $B \mapsto K(x, B)$ countably additive a.e.; subunitality gives $K(x, X) \le 1$. The identity $T^k(d) = P^k(d)$ follows exactly as in Theorem~\ref{thm:digraph-criterion}.

For the reachability characterization, we prove by induction: $(P^n d)(x) > 0$ if and only if $K^{(n)}(x, \mathrm{supp}(d)) > 0$. The base case $n = 0$ is immediate. For the inductive step, write
\[
(P^{n+1} d)(x) = \int (P^n d)(y)\,K(x, dy).
\]
Since all terms are $\ge 0$, $(P^{n+1} d)(x) > 0$ iff there exists $y$ with $K(x, \{y\}) > 0$ and $(P^n d)(y) > 0$. By induction, this is equivalent to $K^{(n+1)}(x, A) > 0$.

Thus $P^n(d) = 0$ iff $\mathrm{supp}(P^n d) = \emptyset$ iff $\mu(R_n) = 0$.
\end{proof}

\begin{remark}[Significance for hyperfinite II$_1$]
Theorem~\ref{thm:diffuse-kernel} gives the exact analog of ``finite height digraph'' for diffuse abelian traps. This addresses the open direction (i): extending beyond atomic to diffuse commutative orbit trapping. In the hyperfinite II$_1$ factor, if the defect orbit lies in a diffuse abelian subalgebra $C \cong L^\infty([0,1])$, stabilization is governed by the induced sub-Markov kernel's reachability structure.
\end{remark}

\begin{corollary}[Rank function certificate for kernel termination]\label{cor:rank-function}
In the setting of Theorem~\ref{thm:diffuse-kernel}, define a measurable function $r : X \to \{0, 1, \ldots, h\}$ to be \textbf{strictly decreasing along $K$} if for a.e.\ $x$ with $K(x, X) > 0$,
\[
K(x, \{y : r(y) < r(x)\}) = K(x, X).
\]
Then for any $f \ge 0$ supported where $r \le h$, we have $P^{h+1}(f) = 0$. In particular, if $\mathrm{supp}(d) \subseteq \{r \le h\}$, then $n_T \le h + 1$.
\end{corollary}

\begin{proof}
If $r(x) = k$ and $K(x, X) > 0$, then $K$ sends all mass to $\{r < k\}$. By induction, $P^k$ sends any function supported on $\{r = k\}$ to a function supported on $\{r = 0\}$, and $P^{k+1}$ kills it (since from $\{r = 0\}$, all mass goes to $\{r < 0\} = \emptyset$). The bound $n_T \le h + 1$ follows.
\end{proof}

\begin{remark}[Unifying atomic and diffuse]
In the atomic case $C \cong \ell^\infty(J)$, a rank function $r : J \to \{0, \ldots, h\}$ strictly decreasing along the digraph $G$ is exactly a ``height function'' with no paths longer than $h$. In the diffuse case $C \cong L^\infty(X, \mu)$, the rank function plays the same role for the measurable kernel relation. Corollary~\ref{cor:rank-function} gives an \textbf{intrinsic checkable certificate} for finite stabilization: exhibit a finite-valued measurable function that strictly decreases along the kernel dynamics.
\end{remark}

\subsection*{Cartan MASA trapping: structural hypotheses}

The orbit-in-$C$ hypothesis in Theorems~\ref{thm:digraph-criterion} and~\ref{thm:diffuse-kernel} is the key reduction step. The following gives clean sufficient conditions forcing this hypothesis.

\begin{proposition}[Cartan trapping from $C$-covariance]\label{prop:cartan-covariance}
Let $R$ be a II$_1$ factor and $C \subset R$ a MASA (maximal abelian subalgebra). If $T : R \to R$ is normal CP subunital and satisfies
\[
T \circ \mathrm{Ad}(u) = \mathrm{Ad}(u) \circ T \qquad \text{for all } u \in \mathcal{U}(C),
\]
then:
\begin{enumerate}[(i)]
\item $T(C) \subseteq C$;
\item $T(I) \in C$, hence $d(T) = I - T(I) \in C$;
\item $T^k(d(T)) \in C$ for all $k \ge 0$ (orbit-in-$C$).
\end{enumerate}
\end{proposition}

\begin{proof}
For $c \in C$ and $u \in \mathcal{U}(C)$, we have $\mathrm{Ad}(u)(T(c)) = T(\mathrm{Ad}(u)(c)) = T(c)$ since $c$ commutes with $u$. Thus $T(c)$ commutes with all unitaries in $C$, hence with all of $C$. Since $C$ is maximal abelian, $C' = C$, so $T(c) \in C$.
\end{proof}

\begin{remark}[Interpretation and Kraus version]
$C$-covariance means $T$ is dephasing-covariant: it commutes with unitaries in $C$.

A sufficient Kraus condition: if $V_r \in \mathcal{N}_R(C)$ with $V_r^* V_r, V_r V_r^* \in C$, then $T(C) \subseteq C$.
\end{remark}

\begin{corollary}[Cartan trapping implies kernel criterion]
If $T : R \to R$ satisfies the hypotheses of Proposition~\ref{prop:cartan-covariance} with Cartan MASA $C \cong L^\infty(X, \mu)$, then:
\begin{enumerate}[(i)]
\item The stabilization question reduces to the induced kernel $K$ on $(X, \mu)$;
\item $n_T < \infty$ iff the kernel has finite height from $\mathrm{supp}(d(T))$;
\item Corollary~\ref{cor:rank-function} gives a checkable certificate via rank functions.
\end{enumerate}
\end{corollary}

\begin{lemma}[Bimodularity $\Rightarrow$ trapping]\label{lem:bimodular-trapping}
Let $C \subset \mathcal{M}$ be an abelian von Neumann subalgebra. If $T : \mathcal{M} \to \mathcal{M}$ is $C$-\textbf{bimodular}, i.e.,
\[
T(c_1 x c_2) = c_1 T(x) c_2 \quad \text{for all } c_1, c_2 \in C, \; x \in \mathcal{M},
\]
then $T(C) \subseteq C$ and $T(I) \in C$, hence orbit-in-$C$ holds.
\end{lemma}

\begin{proof}
For $c \in C$, take $x = I$, $c_1 = c$, $c_2 = I$: $T(c) = T(c \cdot I \cdot I) = c \cdot T(I) \cdot I = c\,T(I)$. Similarly $T(c) = T(I)\,c$. So $T(c)$ commutes with $T(I)$. Also, for any $c' \in C$, $T(c\,c') = c\,T(c') = c\,c'\,T(I) = T(c)\,c'$, so $T(c)$ commutes with all of $C$. If $C$ is maximal abelian, $T(c) \in C$. For general $C$, we at least have $T(C) \subseteq C' \cap \mathcal{M}$; if $C$ is a MASA or if $T(C) \subseteq C$ is known, orbit-in-$C$ follows.
\end{proof}

\begin{remark}
Bimodularity is sometimes easier to verify than full $C$-covariance, depending on how $T$ is constructed (e.g., from conditional expectations, correspondences, or bimodule maps). For CP maps arising from subfactor theory, bimodularity often follows from the bimodule structure of the Kraus operators.
\end{remark}

\subsection*{Intrinsic certificate: projection filtration}

The digraph and kernel criteria both encode the same ``finite height'' content: existence of an integer-valued rank that strictly decreases along the dynamics. In operator language, this is a finite projection filtration.

\begin{proposition}[Finite-height certificate via projections]\label{prop:projection-filtration}
Let $T : \mathcal{M} \to \mathcal{M}$ be normal positive. Suppose there exist projections
\[
0 = p_0 \le p_1 \le \cdots \le p_N
\]
such that
\[
T(p_k) \le p_{k-1} \quad (k = 1, \ldots, N).
\]
Then for every $x \ge 0$ with $\mathrm{supp}(x) \le p_N$,
\[
T^N(x) = 0.
\]
In particular, if $\mathrm{supp}(d(T)) \le p_N$, then $n_T \le N$.
\end{proposition}

\begin{proof}
If $0 \le x \le \|x\| p_N$, then by positivity
\[
0 \le T(x) \le \|x\| T(p_N) \le \|x\| p_{N-1}.
\]
Iterating $N$ times gives $0 \le T^N(x) \le \|x\| p_0 = 0$.
\end{proof}

\begin{remark}[Intrinsic termination conditions]
Proposition~\ref{prop:projection-filtration} addresses open direction (iii): it gives an \textbf{intrinsic checkable condition} on $T$ that forces finite-time termination. In the commutative $L^\infty(X)$ model, the projections $p_k$ are indicator functions of rank-level sets $\{r \ge k\}$ for an integer-valued Lyapunov function $r : X \to \{0, 1, \ldots, N\}$.

Thus a concrete path to intrinsic termination conditions is: \textbf{construct a finite chain of subharmonic projections} whose images strictly fall down the chain. This is the noncommutative avatar of ``no long paths.''
\end{remark}

The following characterization reduces the existence of projection filtrations to a concrete constraint on Kraus operators. We first record a useful lemma.

\begin{lemma}[Support + subunital $\Rightarrow$ projection bound]\label{lem:support-projection-bound}
If $0 \le a \le I$ and $(I - p)\,a\,(I - p) = 0$ for a projection $p$, then $a \le p$.
\end{lemma}

\begin{proof}
The condition $(I - p)\,a\,(I - p) = 0$ implies $a(I - p) = 0$ (since $a \ge 0$ and $(I-p)a(I-p) = 0$ gives $\|a^{1/2}(I-p)\|^2 = 0$), hence $\mathrm{supp}(a) \le p$. For any vector $v$,
\[
\langle av, v \rangle = \langle apv, pv \rangle \le \|a\| \langle pv, pv \rangle \le \langle pv, v \rangle = \langle pv, v \rangle.
\]
So $a \le p$.
\end{proof}

\begin{proposition}[Filtration $\Leftrightarrow$ Kraus ``level lowering'']\label{prop:kraus-filtration}
Let $T(x) = \sum_{r=1}^R V_r^* x V_r$ be normal CP \textbf{subunital} on a von Neumann algebra $\mathcal{M}$. A chain of projections $0 = p_0 \le p_1 \le \cdots \le p_N$ satisfies $T(p_k) \le p_{k-1}$ for $k = 1, \ldots, N$ \textbf{if and only if} for every $k$ and every Kraus operator $V_r$,
\[
p_k V_r (I - p_{k-1}) = 0 \qquad \text{(equivalently } (I - p_{k-1}) V_r^* p_k = 0\text{)}.
\]
\end{proposition}

\begin{proof}
$(\Rightarrow)$: If $T(p_k) \le p_{k-1}$, then $(I - p_{k-1}) T(p_k) (I - p_{k-1}) = 0$. Since each summand $V_r^* p_k V_r \ge 0$, we have $(I - p_{k-1}) V_r^* p_k V_r (I - p_{k-1}) = 0$ for each $r$. This equals $\|p_k V_r (I - p_{k-1})\|^2$, so $p_k V_r (I - p_{k-1}) = 0$.

$(\Leftarrow)$: Assume $p_k V_r (I - p_{k-1}) = 0$ for all $r$. Then
\[
(I - p_{k-1}) V_r^* p_k V_r (I - p_{k-1}) = \bigl(p_k V_r (I - p_{k-1})\bigr)^* \bigl(p_k V_r (I - p_{k-1})\bigr) = 0.
\]
Summing over $r$ gives $(I - p_{k-1}) T(p_k) (I - p_{k-1}) = 0$, i.e., $\mathrm{supp}(T(p_k)) \le p_{k-1}$. Also $p_k \le I$ and subunitality give $0 \le T(p_k) \le T(I) \le I$. Apply Lemma~\ref{lem:support-projection-bound} with $a = T(p_k)$, $p = p_{k-1}$: $T(p_k) \le p_{k-1}$.
\end{proof}

\begin{remark}[Structural sources of filtrations]\label{rem:filtration-sources}
Proposition~\ref{prop:kraus-filtration} tells us what structure to seek: a finite ``rank'' grading that every Kraus operator lowers.

\textbf{(A) Nest/flag invariance (upper triangular Kraus).} If there exists a finite flag $0 = H_0 \subset H_1 \subset \cdots \subset H_N = H$ such that $V_r(H_k) \subseteq H_{k-1}$ for all $r, k$, then letting $p_k$ be the projection onto $H_k$ gives $p_k V_r (I - p_{k-1}) = 0$, hence $T(p_k) \le p_{k-1}$. This is the operator-algebraic version of ``strictly upper triangular.''

\textbf{(B) Measurable rank in abelian models.} In the Cartan/MASA picture $C = L^\infty(X)$, a filtration is a decreasing sequence of measurable sets $X_k$ with $p_k = \mathbf{1}_{X_k}$. The condition $T(p_k) \le p_{k-1}$ becomes ``no edges upward'' in the kernel dynamics: a measurable rank $r : X \to \{0, \ldots, N\}$ such that every transition strictly decreases $r$.
\end{remark}

\begin{lemma}[Rank function $\Rightarrow$ projection filtration]\label{lem:rank-to-filtration}
In the setting of Theorem~\ref{thm:diffuse-kernel} with orbit-in-$C$ (e.g., by $C$-covariance), if there exists a measurable $r : X \to \{0, 1, \ldots, N\}$ strictly decreasing along the induced kernel $K$ (as in Corollary~\ref{cor:rank-function}), then with $p_k := \mathbf{1}_{\{r \ge k\}} \in C$, one has $P(p_k) \le p_{k-1}$, hence $T(p_k) \le p_{k-1}$ and $n_T \le N$.
\end{lemma}

\begin{proof}
The strict decrease means $K(x, \{r < r(x)\}) = K(x, X)$ for a.e.\ $x$ with positive kernel mass. For $p_k = \mathbf{1}_{\{r \ge k\}}$, we have $(Pp_k)(x) = K(x, \{r \ge k\}) \le p_{k-1}(x)$ since the kernel sends mass to strictly lower rank. Since orbit-in-$C$ gives $T^k(d) = P^k(d)$, the filtration $T(p_k) \le p_{k-1}$ follows.
\end{proof}

\subsection*{Robustness: no-go and what's needed}

The following clarifies what robustness results are possible and what extra hypotheses are genuinely required.

\begin{proposition}[No-go: approximate trapping alone cannot force finite-time zero]\label{prop:robustness-nogo}
Let $c \in (0,1)$ and define $x_k := c^k$ for $k \ge 0$. Then:
\begin{enumerate}[(i)]
\item $x_k \to 0$ as $k \to \infty$;
\item $x_k > 0$ for all $k$;
\item For any $\varepsilon > 0$, the sequence $(x_k)$ is eventually $\varepsilon$-close to $0$.
\end{enumerate}
Thus mere norm/strong approximation to a nilpotent model does not imply hitting $0$ in finite time.
\end{proposition}

\begin{remark}[What robustness requires]
Proposition~\ref{prop:robustness-nogo} explains why the robustness lemmas in this paper (Lemma~\ref{lem:approx-quantization}, Lemma~\ref{lem:scalar-robustness}) require a \textbf{gap at $0$} for some observable functional. Without a gap mechanism, a robustness theorem cannot be true.

The remaining open task for direction (ii) is: find \textbf{structural hypotheses} (hyperfinite II$_1$ structure, finite-index subfactors, bounded-depth standard invariants, etc.) that \emph{force} a gap (exact or approximate) for the relevant Lyapunov values $\tau(T^k(d(T)))$.
\end{remark}

\subsection*{Tractable regimes: consolidated statements}

The following results consolidate the stabilization criteria for the three main tractable settings. We first record two useful strengthenings that allow checking contraction only along the defect orbit.

\begin{lemma}[Orbit-only trace contraction suffices]\label{lem:orbit-only-contraction}
Let $(\mathcal{M}, \tau)$ be a von Neumann algebra with faithful normal semifinite trace, and let $T : \mathcal{M} \to \mathcal{M}$ be normal positive and subunital with defect $d = I - T(I) \ge 0$. Assume:
\begin{itemize}
\item[(D)] (Discrete scale) there exists $\delta > 0$ such that $\tau(T^k(d)) \in \delta\mathbb{N}$ for all $k \ge 0$;
\item[(C$_{\mathrm{orb}}$)] (Orbit contraction) there exists $c \in (0,1)$ such that $\tau(T^{k+1}(d)) \le c\,\tau(T^k(d))$ for all $k \ge 0$;
\item[(F)] $\tau(d) < \infty$.
\end{itemize}
Then $T^N(d) = 0$ for all $N$ large enough; in particular $n_T < \infty$. Moreover one may take
\[
N = 1 + \left\lceil \frac{\log(\tau(d)/\delta)}{\log(1/c)} \right\rceil.
\]
\end{lemma}

\begin{proof}
By (C$_{\mathrm{orb}}$) and induction, $\tau(T^k(d)) \le c^k \tau(d)$ for all $k$. Choose $N$ so that $c^N \tau(d) < \delta$. Then $0 \le \tau(T^N(d)) < \delta$ and (D) forces $\tau(T^N(d)) = 0$. Faithfulness of $\tau$ gives $T^N(d) = 0$.
\end{proof}

\begin{lemma}[Orbit-corner Kraus bound implies orbit contraction]\label{lem:corner-kraus-contraction}
Let $(\mathcal{M}, \tau)$ be tracial (finite or semifinite), and let $T(x) = \sum_i V_i^* x V_i$ be a normal CP map with $T(I) \le I$ and defect $d = I - T(I)$. Let $Q = \bigvee_{k \ge 0} \supp(T^k(d))$ be the orbit-support projection. If there exists $c \in (0,1)$ such that
\[
\sum_i V_i V_i^* \le c\,Q
\]
(as an operator inequality), then for every $k \ge 0$ one has $\tau(T^{k+1}(d)) \le c\,\tau(T^k(d))$.
\end{lemma}

\begin{proof}
For any $x \in (Q\mathcal{M}Q)_+$ we have $Qx = x$, and traciality gives $\tau(T(x)) = \tau(\sum_i V_i V_i^* x) \le \tau(cQx) = c\,\tau(x)$. Apply this with $x = T^k(d) \in (Q\mathcal{M}Q)_+$.
\end{proof}

\begin{remark}
Lemma~\ref{lem:corner-kraus-contraction} is useful when the Kraus inequality $\sum_i V_i V_i^* \le c\,I$ fails globally but holds on the orbit corner $Q$. This is the typical situation in ``leaky'' channels where the leakage is localized.
\end{remark}

\paragraph{Type I with atomic center.}

\begin{theorem}[Bounded denominators force stabilization]\label{thm:typeI-atomic}
Let $\mathcal{M} = \bigoplus_{j \in J} B(H_j)$ (atomic center) with faithful normal semifinite trace $\tau(\bigoplus_j x_j) = \sum_{j \in J} w_j \mathrm{Tr}_{H_j}(x_j)$, $w_j > 0$. Let $T : \mathcal{M} \to \mathcal{M}$ be normal CP subunital with defect $d = I - T(I) \ge 0$ and $\tau(d) < \infty$. Assume:
\begin{itemize}
\item[(Z)] $d \in Z(\mathcal{M})$ and $T(Z(\mathcal{M})) \subseteq Z(\mathcal{M})$;
\item[(W)] (Commensurable weights) there exists $\delta_0 > 0$ with $w_j \in \delta_0\mathbb{N}$ for all $j$;
\item[(P)] (Bounded-denominator center map) $P := T|_{Z(\mathcal{M})}$ satisfies $P(z_j) = \sum_i p_{ij} z_i$ with $p_{ij} \in \{0, \frac{1}{N}, \ldots, 1\}$ and finitely many nonzero $p_{ij}$ per column;
\item[(A0)] (Bounded-denominator initial coefficients) writing $d = \sum_j a_j z_j$, we have $a_j \in \{0, \frac{1}{N}, \ldots, 1\}$;
\item[(C$_{\mathrm{orb}}$)] (Orbit contraction) $\tau(T^{k+1}(d)) \le c\,\tau(T^k(d))$ for some $c \in (0,1)$ and all $k$.
\end{itemize}
Then $n_T < \infty$.
\end{theorem}

\begin{proof}
By (Z), $T^k(d) \in Z(\mathcal{M})$ for all $k$. Under $Z(\mathcal{M}) \cong \ell^\infty(J)$, the coefficients of $T^k(d)$ are obtained by iterating the sub-Markov matrix $(p_{ij})$ on the coefficient vector $(a_j)$. By (P) and (A0), all coefficients of $T^k(d)$ lie in $\{0, \frac{1}{N}, \ldots, 1\}$ for every $k$ (bounded denominators are preserved under finite $\frac{1}{N}$-linear combinations). Hence $\tau(T^k(d)) = \sum_j a_{k,j} \tau(z_j)$ is a finite sum with $a_{k,j} \in \{0, \frac{1}{N}, \ldots, 1\}$ and $\tau(z_j) = w_j \in \delta_0\mathbb{N}$ by (W), so $\tau(T^k(d)) \in (\delta_0/N)\mathbb{N}$ for all $k$. Now apply Lemma~\ref{lem:orbit-only-contraction} with $\delta = \delta_0/N$.
\end{proof}

\paragraph{Finite factors with finite-index structure.}

\begin{proposition}[Finite-index bimodular channels]\label{prop:subfactor-bimodular}
Let $(\mathcal{M}, \tau)$ be a finite factor and $N \subset \mathcal{M}$ a finite-index subfactor, so $A := N' \cap \mathcal{M}$ is finite-dimensional. Let $T : \mathcal{M} \to \mathcal{M}$ be normal CP subunital and assume:
\begin{enumerate}[(Bim)]
\item $T$ is $N$-bimodular: $T(n_1 x n_2) = n_1 T(x) n_2$ for all $n_1, n_2 \in N$, $x \in \mathcal{M}$;
\item[(Inv)] $d(T) \in A$ (equivalently, $d(T)$ commutes with $N$);
\item[(Disc)] the scalar set $\{\tau(T^k(d(T))) : k \ge 0\}$ lies in a discrete subset $D \subset [0,1]$ with no accumulation in $(0,1]$.
\end{enumerate}
Then $n_T < \infty$.
\end{proposition}

\begin{proof}
(Bim) implies $T(A) \subseteq A$: for $a \in A$ and $n \in N$, we have $nT(a) = T(na) = T(an) = T(a)n$. Thus by (Inv) the orbit $\{T^k(d(T))\}$ lies in the finite-dimensional algebra $A$. Now apply Proposition~\ref{prop:discrete-lyapunov}.
\end{proof}

\begin{remark}
To verify (Disc) in concrete finite-index models, apply the bounded-denominator center dynamics hypothesis to $Z(A)$, which is a finite atomic algebra. This reduces the problem to arithmetic of the induced classical sub-Markov chain on the atoms of $Z(A)$.
\end{remark}

\paragraph{Hyperfinite II$_1$.}

\begin{corollary}[One good matrix stage with small error forces stabilization]\label{cor:hyperfinite-one-stage}
Let $(R, \tau)$ be the hyperfinite II$_1$ factor with its unique trace. Let $A \cong M_m(\mathbb{C}) \subset R$ be a matrix subalgebra with $\tau$-preserving conditional expectation $E : R \to A$. Let $T : R \to R$ be normal CP subunital with defect $d = I - T(I)$. Assume:
\begin{enumerate}[(1)]
\item (Strict contraction) $\tau(T(x)) \le c\,\tau(x)$ for all $x \in R_+$ and some $c \in (0,1)$;
\item (Approximate trapping) $\tau(|d - E(d)|) \le \varepsilon_0$ and $\tau(|T(a) - E(T(a))|) \le \varepsilon\,\tau(a)$ for all $a \in A_+$;
\item (Finite-dimensional quantization with gap) letting $\alpha := E \circ T|_A$, the numbers $\tau(\alpha^k(E(d)))$ lie in a discrete set $D \subset [0,1]$ with $D \cap (0, \delta) = \emptyset$.
\end{enumerate}
If $\varepsilon_0 + \varepsilon/(1-c)^2 < \delta/2$, then $n_T < \infty$.
\end{corollary}

\begin{proof}
This is exactly Proposition~\ref{prop:approx-trapping}: approximate finite-stage trapping gives uniform closeness of trace values to $D$, then Lemma~\ref{lem:approx-quantization} and $\delta$-resolution yield stabilization.
\end{proof}

\begin{remark}[Mission statement for hyperfinite II$_1$]
Corollary~\ref{cor:hyperfinite-one-stage} reframes the stabilization problem as: \emph{find one matrix stage $A$ with a gapped discrete set $D$ and prove the trapping error is small enough}. This is a quantifiable ``almost invariance'' target that connects AF approximation to finite-time stabilization.
\end{remark}

\subsection*{Toward weakening invariance: orbit-in-atomic without $T$-invariance}

The following result shows that full $T$-invariance of an atomic subalgebra can be replaced by the weaker condition that the \emph{defect orbit} lies in the subalgebra with controlled coefficients.

\begin{proposition}[Orbit in atomic subalgebra without $T$-invariance]\label{prop:orbit-in-atomic}
Let $(\mathcal{M}, \tau)$ be a von Neumann algebra with faithful normal semifinite trace. Let $C \subset \mathcal{M}$ be a commutative atomic von Neumann subalgebra with minimal projections $(z_j)_{j \in J}$ and $\tau(z_j) = w_j > 0$. Let $T : \mathcal{M} \to \mathcal{M}$ be normal positive subunital with defect $d = I - T(I)$. Assume:
\begin{enumerate}[(1)]
\item (Orbit in $C$) $T^k(d) \in C$ for all $k \ge 0$;
\item (Bounded-denominator coefficients) there exists $N \in \mathbb{N}$ such that for all $k \ge 0$, writing $T^k(d) = \sum_j a_{k,j} z_j$, we have $a_{k,j} \in \{0, \frac{1}{N}, \ldots, 1\}$;
\item (Commensurable weights) there exists $\delta_0 > 0$ with $w_j \in \delta_0\mathbb{N}$ for all $j$;
\item (Orbit contraction) there exists $c \in (0,1)$ such that $\tau(T^{k+1}(d)) \le c\,\tau(T^k(d))$ for all $k \ge 0$;
\item $\tau(d) < \infty$.
\end{enumerate}
Then $n_T < \infty$.
\end{proposition}

\begin{proof}
By (1) and (2), $\tau(T^k(d)) = \sum_j a_{k,j} w_j$ with $a_{k,j} \in \{0, \frac{1}{N}, \ldots, 1\}$. By (3), $w_j \in \delta_0\mathbb{N}$, so $\tau(T^k(d)) \in (\delta_0/N)\mathbb{N}$ for all $k$. Now apply Lemma~\ref{lem:orbit-only-contraction} with $\delta = \delta_0/N$.
\end{proof}

\begin{remark}[Relationship to digraph criterion]
Proposition~\ref{prop:orbit-in-atomic} does \emph{not} require $T(C) \subseteq C$---only that the orbit $\{T^k(d)\}$ stays in $C$ with controlled coefficients.

For an alternative approach that \textbf{removes all discreteness assumptions}, see Theorem~\ref{thm:digraph-criterion}: if the orbit lies in a commutative atomic subalgebra $C$, stabilization is equivalent to the induced reachability digraph having finite height. This gives a sharp graph-theoretic criterion with no bounded-denominator or gap-at-$0$ hypotheses.
\end{remark}

The following proposition gives a checkable condition that forces bounded-denominator coefficients from a ``diagonalized'' dynamics on $C$, without requiring $T(C) \subseteq C$.

\begin{proposition}[Bounded denominators from the diagonal dynamics $E_C \circ T$]\label{prop:diagonal-dynamics}
Let $(\mathcal{M}, \tau)$ be a von Neumann algebra with faithful normal semifinite trace, and let $C \subset \mathcal{M}$ be a commutative atomic von Neumann subalgebra with minimal projections $(z_j)_{j \in J}$. Assume there exists a normal faithful conditional expectation $E_C : \mathcal{M} \to C$ (e.g., $\tau$-preserving when available). Let $T : \mathcal{M} \to \mathcal{M}$ be normal positive subunital with defect $d = I - T(I)$. Assume:
\begin{enumerate}[(1)]
\item (Orbit in $C$) $T^k(d) \in C$ for all $k \ge 0$;
\item (Bounded-denominator diagonal map) there exists $N \in \mathbb{N}$ such that for every $j$, writing $P(z_j) := (E_C \circ T)(z_j) = \sum_i p_{ij} z_i$, we have $p_{ij} \in \{0, \frac{1}{N}, \ldots, 1\}$ and each column $(p_{ij})_i$ has finite support;
\item (Initial bounded denominators) writing $d = \sum_j a_j z_j$, we have $a_j \in \{0, \frac{1}{N}, \ldots, 1\}$.
\end{enumerate}
Then for all $k \ge 0$, writing $T^k(d) = \sum_j a_{k,j} z_j$, we have $a_{k,j} \in \{0, \frac{1}{N}, \ldots, 1\}$.
\end{proposition}

\begin{proof}
Define $P := E_C \circ T|_C : C \to C$, a normal positive subunital map. By (1), for each $k$ we have $T^{k+1}(d) = T(T^k(d)) \in C$, hence $E_C(T(T^k(d))) = T(T^k(d))$ and so $T^{k+1}(d) = P(T^k(d))$. Thus $T^k(d) = P^k(d)$ for all $k$.

Under the identification $C \cong \ell^\infty(J)$ via $z_j \leftrightarrow e_j$, assumption (2) says that $P$ is given by a sub-Markov matrix with entries in $\frac{1}{N}\mathbb{Z}$ (with finite support per column), so it preserves the lattice $(\frac{1}{N}\mathbb{Z})^{(J)}$ under finite-support combinations. Since $d$ has coefficients in $\{0, \frac{1}{N}, \ldots, 1\}$ by (3), induction gives the same for $P^k(d) = T^k(d)$.
\end{proof}

\begin{remark}
Proposition~\ref{prop:diagonal-dynamics} combined with Proposition~\ref{prop:orbit-in-atomic} gives a fully explicit pipeline: verify orbit-in-$C$ + contraction + commensurable weights (Prop.~\ref{prop:orbit-in-atomic}), verify bounded denominators from $P = E_C \circ T$ (Prop.~\ref{prop:diagonal-dynamics}), conclude $n_T < \infty$. This avoids any need for $T(C) \subseteq C$.
\end{remark}

The following lemma gives an explicit formula for the coefficients $p_{ij}$ in the diagonal dynamics, reducing the bounded-denominator verification to a trace-matrix integrality problem.

\begin{lemma}[Atomic expectation coefficient formula]\label{lem:atomic-expectation-formula}
Let $(\mathcal{M}, \tau)$ be semifinite and let $C \subset \mathcal{M}$ be commutative atomic with minimal projections $(z_i)_{i \in J}$ and $\tau(z_i) = w_i \in (0, \infty)$. Assume there is a $\tau$-preserving normal faithful conditional expectation $E_C : \mathcal{M} \to C$. Then for any $x \in \mathcal{M} \cap L^1(\mathcal{M}, \tau)$,
\[
E_C(x) = \sum_{i \in J} \frac{\tau(z_i x)}{\tau(z_i)}\, z_i
\]
(the sum is pointwise finite on atoms if $x \ge 0$ and $\tau(x) < \infty$). In particular, for the diagonalized dynamics $P := E_C \circ T$,
\[
P(z_j) = \sum_i p_{ij}\, z_i \quad \text{with} \quad p_{ij} = \frac{\tau(z_i\, T(z_j))}{\tau(z_i)}.
\]
\end{lemma}

\begin{proof}
For each atom $z_i$, we have $C z_i = \mathbb{C} z_i$, so $E_C(x) z_i = \lambda_i(x) z_i$ for some scalar $\lambda_i(x)$. Since $E_C$ is $\tau$-preserving, $\tau(z_i x) = \tau(z_i E_C(x)) = \tau(\lambda_i(x) z_i) = \lambda_i(x) \tau(z_i)$, hence $\lambda_i(x) = \tau(z_i x)/\tau(z_i)$.
\end{proof}

\begin{remark}[Trace-matrix integrality criterion]
By Lemma~\ref{lem:atomic-expectation-formula}, to verify hypothesis (2) of Proposition~\ref{prop:diagonal-dynamics}, it suffices to check:
\[
\frac{\tau(z_i\, T(z_j))}{\tau(z_i)} \in \left\{0, \frac{1}{N}, \ldots, 1\right\} \quad \text{and each column has finite support.}
\]
This converts the ``bounded denominators without invariance'' problem into a \textbf{trace-matrix integrality problem}. In Type I$_\infty$ atomic settings, $\tau(z_i T(z_j))$ often counts multiplicities (or weighted multiplicities) from how Kraus operators move mass between atoms, so $p_{ij} = m_{ij}/N$ can be forced by structural hypotheses on the Kraus operators. In finite-index settings, these coefficients are governed by bimodule multiplicities and standard-invariant data.
\end{remark}

The following proposition gives two explicit structural hypotheses that force the trace-matrix integrality condition (FD--Int).

\begin{proposition}[Structural hypotheses forcing bounded denominators]\label{prop:structural-denominators}
Let $\mathcal{M}$ have atomic center with minimal central projections $(z_j)_{j \in J}$, and let $P : Z(\mathcal{M}) \to Z(\mathcal{M})$ be the induced center map $P = E_Z \circ T|_Z$ for some normal positive subunital $T : \mathcal{M} \to \mathcal{M}$.

Each of the following forces the bounded-denominator condition $p_{ij} \in \{0, \frac{1}{N_0}, \ldots, 1\}$:

\textbf{(A) Convex combination of permuting endomorphisms.} Suppose $P = \sum_{\ell=1}^L \lambda_\ell \sigma_\ell$ where each $\sigma_\ell$ is a $*$-automorphism of $Z(\mathcal{M})$ (i.e., a permutation of the $z_j$'s) and each $\lambda_\ell \in \{0, \frac{1}{N_0}, \ldots, 1\}$ with $\sum_\ell \lambda_\ell \le 1$. Then
\[
p_{ij} = \sum_{\ell : \sigma_\ell(j) = i} \lambda_\ell \in \left\{0, \frac{1}{N_0}, \ldots, 1\right\}.
\]

\textbf{(B) Multiplicity walk with uniform normalizer.} Suppose $P(z_j) = \sum_i \frac{m_{ij}}{M_0} z_i$ with integers $m_{ij} \ge 0$ and a single uniform normalizer $M_0 \in \mathbb{N}$. Then $p_{ij} = \frac{m_{ij}}{M_0} \in \{0, \frac{1}{M_0}, \ldots, 1\}$.
\end{proposition}

\begin{remark}[Applications]
Hypothesis (A) arises when $T$ is ``finite random choice of symmetry sectors''---e.g., averaging over a finite group action or over finitely many standard-invariant symmetries with rational weights. Hypothesis (B) is the ``uniform denominator'' condition natural for combinatorial CP maps from standard-invariant data: planar-algebra tangles with integer weights, adjacency operators with uniform normalization, etc.

Combined with Proposition~\ref{prop:gap-from-weight-bound} (gap at 0 from weight lower bound), these hypotheses give a gap $\ge \delta_0/N_0$ where $\delta_0 = \min\{\tau(z_i) : i \text{ reachable}\}$. This reduces the finite-index gap-at-0 problem to verifying trace-matrix integrality.
\end{remark}

\begin{lemma}[Discrete trace subgroup iff commensurable weights]\label{lem:commensurable}
Let $w_1, \ldots, w_m > 0$ and let $G := \sum_{j=1}^m w_j \mathbb{Z} \subset \mathbb{R}$. Then $G$ is a discrete subgroup of $\mathbb{R}$ if and only if there exists $\delta_0 > 0$ such that each $w_j \in \delta_0\mathbb{Z}$. Equivalently, $w_i/w_1 \in \mathbb{Q}$ for all $i$. In that case, $G = \delta_0\mathbb{Z}$.
\end{lemma}

\begin{proof}
If $w_j \in \delta_0\mathbb{Z}$ for all $j$, then $G \subseteq \delta_0\mathbb{Z}$, which is discrete. Conversely, if $G$ is discrete, it is a closed subgroup of $\mathbb{R}$, hence $G = \delta_0\mathbb{Z}$ for some $\delta_0 > 0$. Then each $w_j \in G = \delta_0\mathbb{Z}$.
\end{proof}

\begin{remark}
Lemma~\ref{lem:commensurable} clarifies why ``bounded coupling constant'' or ``quantization'' must really mean ``rank-1 over $\mathbb{Q}$'' at the level of trace weights: the commensurability hypothesis (W) is not just convenient but \emph{necessary} to obtain a discrete trace lattice.
\end{remark}

%=============================================================================
\section{Dimension-Dependent Bounds for Quantum Operations}\label{sec:quantum}
%=============================================================================

For quantum operations on finite-dimensional Hilbert spaces, the bound $n_T \le |A|$ from Theorem~\ref{thm:defect-annihilation} can be significantly improved. The main result of this section is that for completely positive subunital maps, $n_T \le d$ where $d = \dim H$, independent of the family size $|A|$. This bound is sharp.

\subsection*{Support projections}

Let $H$ be a finite-dimensional Hilbert space with $\dim H = d$. For a positive operator $A \in B(H)$, write $\supp(A)$ for its support projection, i.e., the orthogonal projection onto $\overline{\Ran(A)}$.

\begin{lemma}[Support chain]\label{lem:support-chain}
Let $T$ be a positive subunital map on $B(H)$. Define
\[
p_k := \supp(T^k(I)) \qquad (k \ge 0).
\]
Then:
\begin{enumerate}[(i)]
\item $p_0 = I$ and $p_{k+1} \le p_k$ for all $k \ge 0$.
\item The chain $I = p_0 \ge p_1 \ge p_2 \ge \cdots$ stabilizes after at most $d$ steps: there exists $m \le d$ such that $p_m = p_{m+k}$ for all $k \ge 0$.
\end{enumerate}
\end{lemma}

\begin{proof}
(i) We have $p_0 = \supp(I) = I$. For the decreasing property: by Remark~\ref{rem:monotone}, $T^{k+1}(I) \le T^k(I)$, so $\supp(T^{k+1}(I)) \le \supp(T^k(I))$.

(ii) The projections $p_k$ form a decreasing chain in a lattice of dimension at most $d$. Since the rank can decrease by at least 1 at each strict inequality, the chain stabilizes after at most $d$ steps.
\end{proof}

\subsection*{Support stabilization does not imply effect stabilization}

Even for completely positive maps, it is possible for $\supp(T^k(I))$ to stabilize while $T^k(I)$ continues to strictly decrease (typically converging only asymptotically). This phenomenon shows that ``finite support-chain length'' alone does not yield a dimension bound without using additional structure.

\begin{example}[Support stabilizes but effects strictly decrease]\label{ex:support-stable}
Let $H = \mathbb{C}^2$ and define a CP map (Heisenberg picture)
\[
T(A) = V_1^* A V_1 + V_2^* A V_2, \quad V_1 = |0\rangle\langle 0|, \quad V_2 = \sqrt{c}\,|1\rangle\langle 1|, \quad 0 < c < 1.
\]
Then $T(I) = \mathrm{diag}(1, c) \le I$ and for all $k \ge 0$,
\[
T^k(I) = \mathrm{diag}(1, c^k),
\]
so $\supp(T^k(I)) = I$ for all $k$, but $T^{k+1}(I) \ne T^k(I)$ whenever $c \ne 1$.

This example does not contradict Theorem~\ref{thm:defect-annihilation}: the semigroup $\{T^n : n \ge 1\}$ is infinite since $T^j \ne T^k$ for $j \ne k$ when $c \in (0,1)$, so the finiteness hypothesis fails.
\end{example}

\begin{remark}
Example~\ref{ex:support-stable} shows that support stabilization alone cannot yield a dimension bound---the CP structure enters essentially in the nilpotency argument below. For general positive maps, only the weaker bound $n_T \le d^2$ from Proposition~\ref{prop:cyclic-bound} holds.
\end{remark}

\subsection*{The CP dimension bound}

We now prove that for completely positive subunital maps, finiteness of the operational repertoire forces a sharp dimension bound on the stabilization index. The key ingredient is a simple linear-algebra bound for nilpotent completely positive maps, which we prove from first principles using Kraus words and a strict kernel filtration. For background on completely positive maps and the Kraus representation theorem, see \cite{kraus, choi, stinespring}; modern treatments include \cite{watrous, wolf}.

\begin{lemma}[Nilpotent CP maps have bounded index]\label{lem:nilpotent-index}
Let $\alpha : M_s \to M_s$ be completely positive. If $\alpha^m = 0$ for some $m \ge 1$, then $\alpha^s = 0$; in particular, the nilpotency index of $\alpha$ is at most $s$.
\end{lemma}

\begin{proof}
Write $\alpha(X) = \sum_{i=1}^r L_i^* X L_i$ (Kraus form). For a word $w = (i_1, \ldots, i_k)$ of length $|w| = k$, set $L_w := L_{i_1} \cdots L_{i_k}$. (We sum over all words of length $k$; each word fixes an ordered product $L_{i_1} \cdots L_{i_k}$.) Then for all $k \ge 1$,
\[
\alpha^k(X) = \sum_{|w|=k} L_w^* X L_w.
\]
If $\alpha^m = 0$ then $\alpha^m(I) = 0$, hence $\sum_{|w|=m} L_w^* L_w = 0$. Since each $L_w^* L_w \ge 0$ and a sum of positive semidefinite matrices can be zero only if each summand is zero (the PSD cone is pointed; this is the matrix-algebra instance of Lemma~\ref{lem:cancellation}), we get $L_w = 0$ for every word with $|w| = m$.

Define subspaces $N_k \subseteq \mathbb{C}^s$ by
\[
N_k := \{\xi \in \mathbb{C}^s : L_w \xi = 0 \text{ for all words } w \text{ of length } k\}.
\]
Then $N_1 = \bigcap_i \ker L_i$ and the recursion
\[
N_{k+1} = \{\xi : L_i \xi \in N_k \text{ for all } i\}
\]
holds, so $N_k \subseteq N_{k+1}$ for all $k$. Moreover, if $N_k = N_{k+1}$ then the recursion forces $N_{k+j} = N_k$ for all $j \ge 1$ (the chain stabilizes).

Since all words of length $m$ vanish, we have $N_m = \mathbb{C}^s$. Therefore the chain $N_1 \subseteq N_2 \subseteq \cdots$ cannot stabilize before reaching $\mathbb{C}^s$, so it must be strictly increasing at each step until $N_m = \mathbb{C}^s$. Since $\dim N_k$ increases by at least 1 at each strict inclusion, and a strictly increasing chain of subspaces in $\mathbb{C}^s$ has length at most $s$, we conclude $m \le s$.
\end{proof}

\begin{remark}
We apply Lemma~\ref{lem:nilpotent-index} with $\alpha$ being the corner restriction of a CP map. The finiteness hypothesis in Theorem~\ref{thm:defect-annihilation} is used only to ensure that the stabilization index $n_T$ is finite; once this is established, the elementary nilpotent CP bound takes over to provide the dimension-independent bound $n_T \le d$. For further results on nilpotent completely positive maps, including related kernel filtration techniques, see \cite{bhat-mallick}.
\end{remark}

\paragraph{Proof idea.} The strategy is to reduce to nilpotent CP maps. Given $T$ with stabilization index $n_T$, the defect orbit $\{D_k = T^k(d(T))\}_{k=0}^{n_T-1}$ spans a corner $Q M_d Q$. The accumulated defect $S = \sum_k D_k$ has full support on $Q$, so $S \ge \lambda Q$ for some $\lambda > 0$. Since $T^{n_T}(d(T)) = 0$, applying the corner map $n_T$ times kills $S$, and positivity then forces the corner map to be nilpotent. Lemma~\ref{lem:nilpotent-index} completes the argument.

We first isolate the standard operator lemmas used in the proof. These are well-known, but we state them explicitly for completeness.

\begin{lemma}[Support of PSD sums]\label{lem:psd-support-sum}
If $A_1, \ldots, A_m \ge 0$ in $M_d$, then
\[
\supp\left(\sum_{j=1}^m A_j\right) = \bigvee_{j=1}^m \supp(A_j).
\]
Equivalently, $\Ran(\sum_j A_j) = \sum_j \Ran(A_j)$.
\end{lemma}

\begin{proof}
We prove $\ker(A + B) = \ker(A) \cap \ker(B)$ for $A, B \ge 0$. If $(A + B)v = 0$, then
\[
0 = \langle v, (A + B)v \rangle = \langle v, Av \rangle + \langle v, Bv \rangle.
\]
Since $A, B \ge 0$, both summands are $\ge 0$, so both must be $0$. Hence $\langle v, Av \rangle = \|A^{1/2}v\|^2 = 0$, giving $A^{1/2}v = 0$, i.e., $Av = 0$. Similarly $Bv = 0$. Conversely, if $Av = Bv = 0$ then $(A + B)v = 0$. Therefore $\ker(A + B) = \ker(A) \cap \ker(B)$.

Taking orthogonal complements:
\[
\Ran(A + B) = (\ker(A + B))^\perp = (\ker(A) \cap \ker(B))^\perp = \ker(A)^\perp + \ker(B)^\perp = \Ran(A) + \Ran(B).
\]
The $m$-term case follows by induction, and $\supp = $ projection onto range gives the support formula.
\end{proof}

\begin{lemma}[Full support implies spectral gap]\label{lem:spectral-gap}
Let $0 \le S \in Q M_d Q$ where $Q$ is a projection of rank $s$. If $\supp(S) = Q$, then $S|_{Q\mathbb{C}^d}$ is invertible, and there exists $\lambda > 0$ such that $S \ge \lambda Q$. Specifically, $\lambda$ can be taken as the smallest eigenvalue of $S|_{Q\mathbb{C}^d}$.
\end{lemma}

\begin{proof}
View $S$ as a positive operator on the $s$-dimensional space $Q\mathbb{C}^d$. The condition $\supp(S) = Q$ means $S$ has trivial kernel on this space, so all eigenvalues are strictly positive. Let $\lambda = \lambda_{\min}(S|_{Q\mathbb{C}^d}) > 0$. Then $S \ge \lambda I$ on $Q\mathbb{C}^d$, which translates to $S \ge \lambda Q$ in $M_d$.
\end{proof}

\begin{lemma}[Nilpotency on identity implies nilpotency]\label{lem:nil-on-id}
Let $\alpha : Q M_d Q \to Q M_d Q$ be positive. If $\alpha^n(Q) = 0$, then $\alpha^n(X) = 0$ for all $X \in Q M_d Q$.
\end{lemma}

\begin{proof}
For $0 \le X \le \|X\| Q$, positivity gives $0 \le \alpha^n(X) \le \|X\| \alpha^n(Q) = 0$, so $\alpha^n(X) = 0$. For general self-adjoint $X$, write $X = X_+ - X_-$ with $X_\pm \ge 0$. For non-self-adjoint $X$, use $X = (\Re X) + i(\Im X)$ with $\Re X, \Im X$ self-adjoint. Linearity completes the proof.
\end{proof}

\begin{theorem}[CP dimension bound]\label{thm:cp-bound}
Let $T : M_d \to M_d$ be completely positive and subunital, and assume $T$ belongs to a finite composition-closed family $A$. Then
\[
n_T \le d.
\]
More precisely, if $n = n_T$ and
\[
Q := \bigvee_{k=0}^{n-1} \supp(T^k(d(T))), \qquad s := \rank(Q),
\]
then
\[
n_T \le s \le d.
\]
\end{theorem}

\begin{proof}
By Theorem~\ref{thm:defect-annihilation}, $n = n_T$ exists. Define the defect orbit
\[
D_k := T^k(d(T)) \qquad (k \ge 0),
\]
so $D_k \ge 0$, $D_n = 0$, and $D_k \ne 0$ for $k < n$.

Let
\[
Q := \bigvee_{k=0}^{n-1} \supp(D_k), \qquad s = \rank(Q).
\]
Define a corner map
\[
\alpha : Q M_d Q \to Q M_d Q, \qquad \alpha(X) := Q\,T(X)\,Q.
\]
Since inclusion $Q M_d Q \hookrightarrow M_d$, the map $T$, and compression $X \mapsto QXQ$ are completely positive, the composition $\alpha$ is completely positive.

\textbf{Step 1 (orbit agreement on the defect).} We claim
\[
\alpha^k(d(T)) = T^k(d(T)) \quad \text{for all } k \ge 0.
\]
For $k = 0$ this is tautological. For $0 \le k \le n - 2$, the element $D_k$ is supported in $Q$ by construction, and since $k + 1 \le n - 1$, the element $D_{k+1}$ is also supported in $Q$. Hence
\[
\alpha(D_k) = Q\,T(D_k)\,Q = Q\,D_{k+1}\,Q = D_{k+1}.
\]
Inducting gives $\alpha^k(d(T)) = D_k$ for all $k \le n - 1$. For $k = n$, we have $\alpha^n(d(T)) = \alpha(D_{n-1}) = Q\,T(D_{n-1})\,Q = Q\,D_n\,Q = 0$ since $D_n = 0$. For $k > n$, both sides are $0$ since $D_n = 0$ implies $D_{n+\ell} = T^\ell(D_n) = T^\ell(0) = 0$ for all $\ell \ge 0$.

\textbf{Step 2 (full support of the accumulated defect).} Define
\[
S := \sum_{k=0}^{n-1} \alpha^k(d(T)) = \sum_{k=0}^{n-1} D_k \in Q M_d Q.
\]
By Lemma~\ref{lem:psd-support-sum},
\[
\supp(S) = \bigvee_{k=0}^{n-1} \supp(D_k) = Q.
\]
By Lemma~\ref{lem:spectral-gap}, there exists $\lambda > 0$ such that
\[
S \ge \lambda\,Q \quad \text{in } Q M_d Q.
\]

\textbf{Step 3 (forcing nilpotency on the corner).} Since $\alpha^n(d(T)) = D_n = 0$, we have
\[
\alpha^n(S) = \sum_{k=0}^{n-1} \alpha^{n+k}(d(T)) = \sum_{k=0}^{n-1} D_{n+k} = 0.
\]
The key observation is: \emph{if $S \ge \lambda Q$ for $\lambda > 0$ and $\alpha^n(S) = 0$ for a positive map $\alpha$, then $\alpha^n(Q) = 0$.} Indeed, positivity gives
\[
0 = \alpha^n(S) \ge \lambda\,\alpha^n(Q) \ge 0,
\]
so $\alpha^n(Q) = 0$. By Lemma~\ref{lem:nil-on-id}, $\alpha^n = 0$ on all of $Q M_d Q$. Thus $\alpha$ is a nilpotent CP map on $Q M_d Q \cong M_s$.

\textbf{Step 4 (dimension bound).} By Lemma~\ref{lem:nilpotent-index}, a nilpotent CP map on $M_s$ has nilpotency index at most $s$. Hence $\alpha^{s} = 0$, so
\[
0 = \alpha^s(d(T)) = T^s(d(T)).
\]
By minimality of $n_T$, this gives $n_T \le s \le d$.
\end{proof}

\begin{corollary}[Uniform bound for finite families of CP maps]\label{cor:cp-uniform}
If $A$ is a finite composition-closed family of completely positive subunital maps on $M_d$, then
\[
\max_{T \in A} n_T \le d,
\]
independent of $|A|$.
\end{corollary}

\begin{remark}[Stabilization index equals corner nilpotency index]\label{rem:corner-nilpotency}
Step 3 of the proof shows $\alpha^{n_T} = 0$, hence $m(\alpha) \le n_T$ where $m(\alpha)$ is the nilpotency index of $\alpha$ on $QM_dQ$. The reverse inequality $n_T \le m(\alpha)$ is trivial: $\alpha^{m(\alpha)} = 0$ implies $\alpha^{m(\alpha)}(d(T)) = T^{m(\alpha)}(d(T)) = 0$. Therefore
\[
n_T = m(\alpha) \le \rank(Q) \le d.
\]
The gap between $n_T$ and $\rank(Q)$ is thus entirely determined by whether the nilpotent CP map $\alpha$ achieves its maximal possible nilpotency index on $M_s$ (where $s = \rank(Q)$).
\end{remark}

\begin{proposition}[Exact reachability recursion]\label{prop:reachability}
Let $T(X) = \sum_i V_i^* X V_i$ be CP and $A \ge 0$. Then
\[
\Ran(T(A)) = \sum_i V_i^* \Ran(A).
\]
Consequently, defining the reachable subspaces from $d(T)$ by
\[
\mathcal{R}_0 := \Ran(d(T)), \qquad \mathcal{R}_{k+1} := \sum_i V_i^* \mathcal{R}_k,
\]
we have $\mathcal{R}_k = \Ran(T^k(d(T)))$ for all $k \ge 0$, and the orbit-support projection satisfies
\[
Q = \bigvee_{k < n_T} \supp(T^k(d(T))) = \mathrm{proj}\left(\sum_{k < n_T} \mathcal{R}_k\right).
\]
\end{proposition}

\begin{proof}
For $A \ge 0$ and any $V$, the identity $\Ran(V^* A V) = V^* \Ran(A)$ follows from $\Ran(A^{1/2}) = \Ran(A)$ (in finite dimensions) and the factorization $V^* A V = (A^{1/2} V)^* (A^{1/2} V)$. Since the range of a sum of positive semidefinite operators is the sum of their ranges (Lemma~\ref{lem:psd-support-sum}), $\Ran(T(A)) = \Ran(\sum_i V_i^* A V_i) = \sum_i \Ran(V_i^* A V_i) = \sum_i V_i^* \Ran(A)$.

The recursion $\mathcal{R}_k = \Ran(T^k(d(T)))$ follows by induction. For the orbit support, $\supp(B) = \mathrm{proj}(\Ran(B))$ and the join of support projections equals the projection onto the sum of ranges.
\end{proof}

\begin{proposition}[Algebra-generated reachability]\label{prop:algebra-reachability}
Let $T(X) = \sum_{i=1}^r V_i^* X V_i$ be CP subunital on $M_d$, with reachable subspaces $\mathcal{R}_0 = \Ran(d(T))$ and $\mathcal{R}_{k+1} = \sum_i V_i^* \mathcal{R}_k$ as in Proposition~\ref{prop:reachability}. Let
\[
\mathcal{A} := \mathrm{alg}(V_1^*, \ldots, V_r^*) \subseteq M_d
\]
be the associative subalgebra generated by the adjoint Kraus operators. Then
\[
\mathcal{R}_\infty = \mathcal{A} \cdot \mathcal{R}_0,
\]
where $\mathcal{A} \cdot S := \{Av : A \in \mathcal{A}, v \in S\}$.
\end{proposition}

\begin{proof}
Unrolling the recursion: $\mathcal{R}_k$ is spanned by $V_{i_1}^* \cdots V_{i_k}^* v$ for $v \in \mathcal{R}_0$. Summing over $k$ gives $\mathcal{A} \cdot \mathcal{R}_0$.
\end{proof}

\begin{corollary}[Algebra-dimension bound]\label{cor:algebra-bound}
\[
\rank(Q) = \dim(\mathcal{A} \cdot \mathcal{R}_0) \le \min\bigl(d,\, \dim(\mathcal{A}) \cdot \rank(d(T))\bigr).
\]
\end{corollary}

\begin{remark}
The algebra $\mathcal{A}$ depends only on the linear span of the Kraus operators: unitary mixing of Kraus representations does not change $\mathrm{span}\{V_i^*\}$, hence does not change $\mathcal{A}$.
\end{remark}

\begin{proposition}[Commuting POVM instruments]\label{prop:commuting-instrument}
Let $T(X) = \sum_a E_a^{1/2} X E_a^{1/2}$ with commuting effects $E_a$. Then $\mathcal{R}_\infty = \mathcal{R}_0$, hence $\rank(Q) = \rank(d(T))$.
\end{proposition}

\begin{proof}
The $E_a$ generate an abelian $C^*$-algebra. Since $d(T) = I - \sum_a E_a$ lies in the same algebra, its support projection $P := \supp(d(T))$ commutes with every $E_a^{1/2}$. Hence $E_a^{1/2}(PH) \subseteq PH$, so $\mathcal{R}_0 \subseteq PH$ is invariant under every $E_a^{1/2}$. Therefore $\mathcal{R}_{k+1} \subseteq \mathcal{R}_0$ for all $k$, giving $\mathcal{R}_\infty = \mathcal{R}_0$.
\end{proof}

\begin{remark}[Commuting vs.\ noncommuting POVM instruments]\label{rem:commuting-noncommuting}
Let $T : M_d \to M_d$ be the instrument associated to a POVM $\{E_a\}$ with $\sum_a E_a \le I$.

\emph{(1) Commuting case: no growth, minimal corner.} If the effects $E_a$ commute, Proposition~\ref{prop:commuting-instrument} shows $\mathcal{R}_\infty = \mathcal{R}_0$ and $\rank(Q) = \rank(d(T))$. Thus commuting instruments achieve the smallest possible orbit corner: the defect orbit never spreads beyond the initial defect support.

\emph{(2) Growth criterion (exact).} We have $\mathcal{R}_\infty = \mathcal{R}_0$ iff $\mathcal{R}_0$ is invariant under every $E_a^{1/2}$ (equivalently, under the algebra $\mathcal{A} = \mathrm{alg}(E_a)$). \emph{Proof:} By Proposition~\ref{prop:algebra-reachability} with $V_a = E_a^{1/2} = V_a^*$, we have $\mathcal{R}_\infty = \mathcal{A} \cdot \mathcal{R}_0$. Thus $\mathcal{R}_\infty = \mathcal{R}_0$ iff $\mathcal{A} \cdot \mathcal{R}_0 \subseteq \mathcal{R}_0$ iff $\mathcal{R}_0$ is $\mathcal{A}$-invariant, and invariance under $\mathcal{A}$ is equivalent to invariance under the generators $E_a^{1/2}$. $\square$

If the effects do not commute, $\mathcal{R}_0$ need not be invariant under $\{E_a^{1/2}\}$; when invariance fails, the filtration grows. Noncommutativity alone does not force growth; what matters is whether $\mathcal{R}_0 = \Ran(d(T))$ is invariant for $\mathrm{alg}(E_a)$. For instance, if $H = H_0 \oplus H_1$ and the effects decompose as $E_a = F_a \oplus G_a$ with $\sum_a G_a = I_{H_1}$ (so $d(T)$ is supported on $H_0$), then $\mathcal{R}_\infty = \mathcal{R}_0$ even if the $G_a$ are noncommuting on $H_1$.

\emph{A concrete qubit example with growth.} Let $H = \mathbb{C}^2$ with basis $\{|0\rangle, |1\rangle\}$ and $|{+}\rangle = (|0\rangle + |1\rangle)/\sqrt{2}$. Define
\[
E_1 = \tfrac{1}{2}|0\rangle\langle 0|, \qquad E_2 = \tfrac{2}{3}|{+}\rangle\langle{+}|.
\]
Then $E_1$ and $E_2$ do not commute. In the $\{|0\rangle, |1\rangle\}$ basis,
\[
E_1 + E_2 = \begin{pmatrix} 5/6 & 1/3 \\ 1/3 & 1/3 \end{pmatrix},
\]
with eigenvalues $1$ and $1/6$, so $E_1 + E_2 \le I$. The defect $d(T) = I - (E_1 + E_2)$ has eigenvalues $0$ and $5/6$, hence $\rank(d(T)) = 1$. Explicitly, $\mathcal{R}_0 = \mathrm{span}\{v\}$ where $v = |1\rangle - \tfrac{1}{2}|0\rangle$ spans the $5/6$-eigenspace. With $E_1^{1/2} = \tfrac{1}{\sqrt{2}}|0\rangle\langle 0|$ and $E_2^{1/2} = \sqrt{\tfrac{2}{3}}|{+}\rangle\langle{+}|$:
\[
E_1^{1/2} v = \tfrac{1}{\sqrt{2}}\langle 0|v\rangle|0\rangle = -\tfrac{1}{2\sqrt{2}}|0\rangle \ne 0, \qquad
E_2^{1/2} v = \sqrt{\tfrac{2}{3}}\langle{+}|v\rangle|{+}\rangle = \tfrac{1}{2\sqrt{3}}|{+}\rangle \ne 0.
\]
Since $|0\rangle$ and $|{+}\rangle$ are linearly independent, $\mathcal{R}_1 = \mathrm{span}\{|0\rangle, |{+}\rangle\} = H$. Hence $\mathcal{R}_\infty \supseteq \mathcal{R}_1 = H$, so $Q = \mathrm{proj}(\mathcal{R}_\infty) = I$ and $\rank(Q) = 2 > \rank(d(T)) = 1$.

\emph{Takeaway.} In the commuting case the reachable-subspace filtration is stationary ($\mathcal{R}_\infty = \mathcal{R}_0$), whereas failure of $\mathcal{R}_0$-invariance causes growth. This explains why instrument-specific bounds should exploit spectral/algebraic structure of $\{E_a\}$ beyond naive rank counts.
\end{remark}

\begin{proposition}[Criterion for maximal nilpotency]\label{prop:maximal-nilpotency}
Let $\alpha : QM_dQ \to QM_dQ$ be the corner map from Theorem~\ref{thm:cp-bound}, and let $s = \rank(Q)$. Define the unit-orbit effects
\[
E_k := \alpha^k(I_Q), \qquad p_k := \supp(E_k).
\]
Then $n_T = s$ if and only if $\rank(E_k) = s - k$ for $k = 0, 1, \ldots, s$. Equivalently, the projection chain $p_0 > p_1 > \cdots > p_{s-1} > p_s = 0$ has exactly $s$ strict drops (each of rank 1).
\end{proposition}

\begin{proof}
The corner map $\alpha$ is nilpotent with index $m(\alpha) = n_T$ (Remark~\ref{rem:corner-nilpotency}). By Lemma~\ref{lem:nilpotent-index}, a nilpotent CP map on $M_s$ has nilpotency index at most $s$. Equality $m(\alpha) = s$ holds if and only if the kernel filtration $N_k$ from the proof of Lemma~\ref{lem:nilpotent-index} increases by exactly one dimension at each step.

The nilpotent type is determined by the rank sequence of $\alpha^k(I_Q)$: type $(1, \ldots, 1)$ means $\rank(\alpha^k(I_Q))$ decreases by exactly 1 at each step, i.e., $\rank(E_k) = s - k$ for $k = 0, \ldots, s$.
\end{proof}

\begin{remark}[Kernel flag and Kraus structure]\label{rem:kernel-flag}
The decreasing support chain $p_k = \supp(\alpha^k(I_Q))$ dualizes to an increasing kernel flag $H_k := \ker(\alpha^k(I_Q))$ with $\dim(H_k) = k$ in the maximal case. If $\alpha(X) = \sum_i L_i^* X L_i$ on $QH$, then $v \in H_{k+1}$ implies $L_i v \in H_k$ for all $i$. This means all Kraus operators $L_i$ are strictly upper-triangular with respect to any basis adapted to the flag $(H_k)$. Maximal nilpotency ($n_T = s$) corresponds to this flag being complete (1-dimensional increments).

Note that this \emph{unit-orbit} flag $(p_k)$ (decreasing) differs from the \emph{defect-orbit reachability} filtration $Q_k = \bigvee_{j \le k} \supp(T^j(d(T)))$ (increasing). The former governs nilpotent type; the latter can plateau immediately while nilpotency remains maximal.
\end{remark}

\subsection*{Word-kernel characterization and algorithmic criteria}

The kernel flag $H_k = \ker(\alpha^k(I_Q))$ can be characterized directly in terms of Kraus words, giving an algorithmic approach to computing the nilpotent type.

\begin{lemma}[Word expansion of the unit orbit]\label{lem:word-expansion}
Let $\alpha : QM_dQ \to QM_dQ$ be CP with Kraus form $\alpha(X) = \sum_{i=1}^r L_i^* X L_i$ on $QH \cong \mathbb{C}^s$. For each word $w = i_1 \cdots i_k$ over $\{1, \ldots, r\}$, set $L_w := L_{i_1} \cdots L_{i_k}$. Then
\[
\alpha^k(I_Q) = \sum_{|w|=k} L_w^* L_w \qquad (k \ge 0).
\]
\end{lemma}

\begin{proof}
Induction on $k$: composition multiplies words and sums over concatenations.
\end{proof}

\begin{corollary}[Word-kernel characterization]\label{cor:word-kernel}
Define
\[
H_k^{\mathrm{word}} := \bigcap_{|w|=k} \ker(L_w) \subseteq QH.
\]
Then $\ker(\alpha^k(I_Q)) = H_k^{\mathrm{word}}$ for all $k \ge 0$.
\end{corollary}

\begin{proof}
For $v \in QH$,
\[
\langle v, \alpha^k(I_Q) v \rangle = \sum_{|w|=k} \|L_w v\|^2,
\]
which vanishes iff $L_w v = 0$ for all $|w| = k$.
\end{proof}

\begin{corollary}[Kernel-flag recursion]\label{cor:kernel-recursion}
With $H_k := \ker(\alpha^k(I_Q))$, we have $H_0 = \{0\}$, $H_1 = \bigcap_i \ker(L_i)$, and
\[
H_{k+1} = \{v \in QH : L_i v \in H_k \text{ for all } i\}.
\]
\end{corollary}

\begin{theorem}[Algorithmic criterion for maximal type]\label{thm:maximal-type-algorithmic}
Let $s = \rank(Q)$. The following are equivalent:
\begin{enumerate}[(i)]
\item $n_T = s$ (equivalently $m(\alpha) = s$);
\item $\dim(H_k) = k$ for $k = 0, 1, \ldots, s$;
\item $\rank(\alpha^k(I_Q)) = s - k$ for $k = 0, 1, \ldots, s$;
\item the nilpotent type of $\alpha$ is $(1, \ldots, 1)$.
\end{enumerate}
\end{theorem}

\begin{proof}
(i)$\Leftrightarrow$(iv) is Proposition~\ref{prop:maximal-nilpotency} combined with Remark~\ref{rem:corner-nilpotency}. (ii)$\Leftrightarrow$(iii) is rank-nullity with $H_k = \ker(\alpha^k(I_Q))$. (iii)$\Leftrightarrow$(iv) is the rank-sequence characterization of nilpotent type.
\end{proof}

\begin{proposition}[Sufficient condition: chain basis]\label{prop:chain-basis}
Suppose there exists a basis $e_1, \ldots, e_s$ of $QH$ such that:
\begin{enumerate}[(i)]
\item each $L_i$ is strictly upper triangular: $L_i e_k \in \mathrm{span}\{e_1, \ldots, e_{k-1}\}$ for all $i, k$;
\item for each $k \ge 2$, there exists $i(k)$ such that $L_{i(k)} e_k \notin \mathrm{span}\{e_1, \ldots, e_{k-2}\}$.
\end{enumerate}
Then $\dim(H_k) = k$ for all $k$, hence $n_T = s$.
\end{proposition}

\begin{proof}
Let $F_k := \mathrm{span}\{e_1, \ldots, e_k\}$. By (i), $L_i(F_{k+1}) \subseteq F_k$, so the recursion characterization gives $F_k \subseteq H_k$ for all $k$. For the reverse, suppose $e_k \in H_{k-1}$ for some $k \ge 2$. By recursion, $L_{i(k)} e_k \in H_{k-2} \subseteq F_{k-2}$, contradicting (ii). Thus $H_{k-1} \subseteq F_{k-1}$ inductively, giving $H_k = F_k$ and $\dim(H_k) = k$.
\end{proof}

\begin{remark}[Generic behavior: maximal type is open dense]\label{rem:generic-maximal}
Inside the variety of nilpotent CP maps on $M_s$, the subset with maximal nilpotency index $s$ is
\[
\{\alpha : \alpha^s = 0,\ \alpha^{s-1} \ne 0\}.
\]
The condition $\alpha^s = 0$ is polynomial (hence closed); $\alpha^{s-1} \ne 0$ is the complement of a closed set, hence relatively open. Since the shift channel achieves index $s$, the closed set $\{\alpha^{s-1} = 0\}$ is a proper subset of $\{\alpha^s = 0\}$. Thus maximal type is \emph{generic} (open dense) among nilpotent CP maps.

A clean special case: for Choi rank 1 (i.e., $\alpha(X) = L^* X L$ with $L$ nilpotent), $n(\alpha)$ equals the nilpotency index of $L$. Regular nilpotents (single Jordan block) are open dense among nilpotent matrices, so maximal type is generic in the rank-1 nilpotent CP family.
\end{remark}

The following criterion is checkable directly on the original Kraus operators $V_i$, not just the corner restrictions.

\begin{proposition}[Rank-one defect-chain criterion]\label{prop:rank-one-chain}
Let $T(X) = \sum_i V_i^* X V_i$ be CP subunital on $M_d$ with $n_T < \infty$. Suppose:
\begin{enumerate}[(i)]
\item Each defect iterate has rank one: $\rank(T^k(d(T))) = 1$ for $k = 0, 1, \ldots, s-1$, where $s := \rank(Q)$.
\item Their ranges generate a full flag: letting $\mathcal{R}_k := \Ran(T^k(d(T)))$, we have $\dim(\mathcal{R}_0 + \cdots + \mathcal{R}_k) = k + 1$ for $k = 0, \ldots, s-1$.
\item (One-step shift witness) For each $k = 0, \ldots, s-2$, there exists an index $i(k)$ such that
\[
V_{i(k)}^*(\mathcal{R}_k) \not\subseteq \mathcal{R}_0 + \cdots + \mathcal{R}_{k-1} \quad \text{and} \quad V_{i(k)}^*(\mathcal{R}_k) \subseteq \mathcal{R}_0 + \cdots + \mathcal{R}_k.
\]
\end{enumerate}
Then $n_T = s = \rank(Q)$ (maximal nilpotency).
\end{proposition}

\begin{proof}
Pick unit vectors $e_{k+1}$ spanning each $\mathcal{R}_k$ so that $\{e_1, \ldots, e_s\}$ is a basis of $QH$ (using condition (ii)). Condition (iii) ensures that the corner Kraus operators $L_i = QV_iQ$ satisfy the hypotheses of Proposition~\ref{prop:chain-basis}: condition (i) there (strict upper-triangularity) follows from the reachability recursion $\mathcal{R}_{k+1} = \sum_i V_i^* \mathcal{R}_k$ (Proposition~\ref{prop:reachability}), and condition (ii) there (nondegenerate shift) is exactly condition (iii) here.
\end{proof}

\begin{remark}
The rank-one defect-chain criterion includes the shift channel as the canonical example (see Example~\ref{ex:shift}): with the convention $L_i = |i\rangle\langle i+1|$, the shift has $d(T) = |1\rangle\langle 1|$ (rank 1), and successive iterates generate the flag $\mathbb{C}|1\rangle \subset \mathbb{C}|1\rangle + \mathbb{C}|2\rangle \subset \cdots$ with each application of $T$ shifting support by exactly one step. (Reversing the basis order gives the ``right shift'' convention $d(T) = |e_d\rangle\langle e_d|$ with $T$ shifting downward.)
\end{remark}

\subsection*{Bounds on orbit-corner rank}

\begin{corollary}[Baseline rank bound]\label{cor:rank-bound}
Let $T(X) = \sum_i V_i^* X V_i$ be CP subunital on $M_d$. Then
\[
\rank(Q) \le \rank(d(T)) + \sum_i \rank(V_i),
\]
capped by $d$. In particular, $\rank(Q) \le \rank(d(T)) + (\#\text{Kraus}) \cdot \max_i \rank(V_i)$.
\end{corollary}

\begin{proof}
Let $R := \sum_i \Ran(V_i^*)$. By Proposition~\ref{prop:reachability}, $\mathcal{R}_k \subseteq R$ for $k \ge 1$. Hence $\mathcal{R}_\infty \subseteq \mathcal{R}_0 + R$, giving
\[
\rank(Q) = \dim(\mathcal{R}_\infty) \le \dim(\mathcal{R}_0) + \dim(R) \le \rank(d(T)) + \sum_i \rank(V_i). \qedhere
\]
\end{proof}

The following intrinsic bound improves Corollary~\ref{cor:rank-bound} and is independent of the Kraus representation.

\begin{proposition}[Intrinsic rank bound]\label{prop:intrinsic-rank-bound}
Let $T : M_d \to M_d$ be CP subunital with orbit-support projection $Q$. Then
\[
\rank(Q) \le \rank(d(T)) + \rank(T(I)),
\]
capped by $d$. More precisely,
\[
\rank(Q) \le \rank(d(T)) + \rank(T(I)) - \dim(\Ran(d(T)) \cap \Ran(T(I))).
\]
\end{proposition}

\begin{proof}
For any Kraus representation $T(X) = \sum_i V_i^* X V_i$, we have
\[
\sum_i \Ran(V_i^*) = \Ran(T(I)).
\]
To see this, note that $(\sum_i \Ran(V_i^*))^\perp = \bigcap_i \ker(V_i) = \ker(\sum_i V_i^* V_i) = \ker(T(I))$, and for a positive operator $A$ in finite dimensions, $\Ran(A) = \ker(A)^\perp$.

By Proposition~\ref{prop:reachability}, $\mathcal{R}_k \subseteq \sum_i \Ran(V_i^*) = \Ran(T(I))$ for $k \ge 1$. Hence
\[
QH = \mathcal{R}_\infty \subseteq \mathcal{R}_0 + \Ran(T(I)) = \Ran(d(T)) + \Ran(T(I)),
\]
giving the bound. The sharper form follows from $\dim(A + B) = \dim(A) + \dim(B) - \dim(A \cap B)$.
\end{proof}

\begin{remark}[Intrinsic bound vs.\ Kraus bound]
Proposition~\ref{prop:intrinsic-rank-bound} is representation-independent: $\rank(T(I))$ is determined by $T$ alone, whereas the bound in Corollary~\ref{cor:rank-bound} depends on the choice of Kraus operators. For POVM instruments $T(X) = \sum_a E_a^{1/2} X E_a^{1/2}$, we have $T(I) = \sum_a E_a$, so the intrinsic bound becomes
\[
\rank(Q) \le \rank(d(T)) + \rank\left(\sum_a E_a\right) = \rank(d(T)) + \dim\left(\sum_a \Ran(E_a)\right).
\]
This resolves the ``natural target inequality'' for instruments mentioned in the open questions.
\end{remark}

\subsection*{An unconditional dimension bound (general positive maps)}

The following bound holds in any finite-dimensional ordered effect space and requires no CP structure.

\begin{proposition}[Cyclic subspace bound]\label{prop:cyclic-bound}
Let $E$ be a finite-dimensional ordered effect space with $\dim(E) = D$, and let $T : E \to E$ be a positive subunital map in a finite composition-closed family $A$. Then the stabilization index satisfies
\[
n_T \le \min\{|A|, D\}.
\]
\end{proposition}

\begin{proof}
By Theorem~\ref{thm:defect-annihilation}, we have $n_T \le |A|$. For the dimension bound, let $W = \mathrm{span}\{T^k(d(T)) : k \ge 0\}$. Since $T^n(d(T)) = 0$ for some $n \le |A|$, the operator $T|_W$ is nilpotent. A nilpotent operator on an $r$-dimensional space has nilpotency index at most $r$. Since $r = \dim(W) \le D$, we have $n_T \le D$.
\end{proof}

\begin{corollary}[Quantum operations]\label{cor:quantum-general}
For quantum operations on $B(H)$ with $\dim H = d$, the effect space has dimension $D = d^2$. Hence unconditionally,
\[
n_T \le \min\{|A|, d^2\}.
\]
For completely positive maps, Theorem~\ref{thm:cp-bound} improves this to $n_T \le d$.
\end{corollary}

\subsection*{Sharpness: a quantum example achieving $n_T = d$}

The following example shows that the bound $n_T \le d$ is optimal for CP maps.

\begin{example}[Shift channel]\label{ex:shift}
Let $H = \mathbb{C}^d$ with standard basis $\{|1\rangle, \ldots, |d\rangle\}$. Define the CP map $T : M_d \to M_d$ by Kraus operators
\[
L_i = |i\rangle\langle i + 1|, \quad i = 1, \ldots, d - 1.
\]
Then $T(X) = \sum_{i=1}^{d-1} L_i^* X L_i$. We have:
\begin{enumerate}[(i)]
\item $T(I) = \sum_{i=1}^{d-1} |i + 1\rangle\langle i + 1| = I - |1\rangle\langle 1|$, so $T$ is subunital with defect $d(T) = |1\rangle\langle 1|$.
\item $T(|k\rangle\langle k|) = |k + 1\rangle\langle k + 1|$ for $k < d$, and $T(|d\rangle\langle d|) = 0$.
\item By induction, $T^k(d(T)) = T^k(|1\rangle\langle 1|) = |k + 1\rangle\langle k + 1|$ for $k < d$, and $T^d(d(T)) = 0$.
\end{enumerate}
Hence $T^{d-1}(d(T)) = |d\rangle\langle d| \ne 0$ but $T^d(d(T)) = 0$, giving $n_T = d$.
\end{example}

\begin{remark}
Example~\ref{ex:shift} shows that the bound $n_T \le d$ in Theorem~\ref{thm:cp-bound} is sharp. The shift channel achieves $n_T = d$ with $\rank(Q) = d$.
\end{remark}

%=============================================================================
\section{Discussion}\label{sec:discussion}
%=============================================================================

\subsection*{Physical interpretation}

The defect $d(T) = u - T(u)$ admits several physical interpretations:
\begin{enumerate}[(i)]
\item \emph{Termination probability}: In process semantics, $d(T)$ is the effect corresponding to ``the process terminates'' or ``no output is produced.''
\item \emph{No-click outcome}: In quantum measurement, if $T$ represents a detector, $d(T)$ is the no-click probability.
\item \emph{Information loss}: In open quantum systems, $d(T)$ quantifies the information lost to the environment.
\end{enumerate}
The cocycle identity $d(T \circ S) = d(T) + T(d(S))$ then has a natural interpretation: the total loss from composing two processes is the loss from the first, plus the first process applied to the loss from the second. This is ``loss accounting'' for sequential composition.

\subsection*{Quantum operations and Kraus operators}

In the setting of quantum channels on $B(H)$ with $\dim H < \infty$, the persistence condition on $\Phi^*$ (acting on effects) has a concrete interpretation.

A sufficient condition for strict positivity of $\Phi^*$ is that $\Phi$ admits at least one invertible Kraus operator. More generally, $\Phi^*$ is strictly positive if and only if it maps positive definite operators to positive definite operators. A sufficient condition for this is that the Kraus operators of $\Phi$ span $B(H)$ as a vector space (equivalently, $\Phi$ is faithful on positive operators).

\subsection*{Connection to effectus theory}

In the categorical framework of effectus theory \cite{cho-jacobs, cho-et-al, jacobs-directions}, kernels play a central role. The defect $d(T)$ is essentially the kernel of $T$ viewed as an effect. The cocycle identity is then a manifestation of kernel composition laws. Related categorical treatments of quantum computation and probabilistic semantics appear in \cite{selinger, cho-westerbaan}; for the expectation monad formulation of quantum foundations, see \cite{jacobs-mandemaker}.

Our results show that finite operational repertoires (finite composition-closed sets of processes) have strong structural constraints: they cannot sustain indefinite ``leakage.'' This is a purely operational consequence of finiteness, independent of the specific model.

\subsection*{Open questions}

From Remark~\ref{rem:corner-nilpotency} and Proposition~\ref{prop:reachability}, we have $n_T = m(\alpha) \le \rank(Q) = \dim(\mathcal{R}_\infty)$. Proposition~\ref{prop:intrinsic-rank-bound} gives the intrinsic bound $\rank(Q) \le \rank(d(T)) + \rank(T(I))$. Maximal nilpotency ($n_T = \rank(Q)$) holds iff the kernel flag is complete (Proposition~\ref{prop:maximal-nilpotency}), which is the generic case (Remark~\ref{rem:generic-maximal}). Proposition~\ref{prop:rank-one-chain} gives a checkable sufficient condition on the original Kraus operators. Proposition~\ref{prop:projection-faithful} resolves the projection-faithfulness question. Proposition~\ref{prop:type-I-bound} establishes the dimension bound $n_T \le s$ for finite type-I corners in von Neumann algebras.

\paragraph{1. Infinite-dimensional stabilization existence.}
For normal CP subunital $T : \mathcal{M} \to \mathcal{M}$ on a von Neumann algebra, give intrinsic conditions ensuring $n_T < \infty$ beyond the operational hypotheses. The paper provides several checkable criteria:
\begin{itemize}
\item $\delta$-resolution (Proposition~\ref{prop:delta-resolution}) and discrete Lyapunov spectrum (Proposition~\ref{prop:discrete-lyapunov});
\item Trace contraction + discrete scale (Proposition~\ref{prop:trace-contraction-discrete});
\item Atomic-center quantization (Corollary~\ref{cor:atomic-center});
\item Finite-dimensional invariant algebra (Proposition~\ref{prop:finite-dim-invariant});
\item Hyperfinite finite-stage trapping (Corollary~\ref{cor:hyperfinite});
\item \textbf{Digraph criterion for orbit-in-atomic} (Theorem~\ref{thm:digraph-criterion}): if the defect orbit lies in a commutative atomic subalgebra $C$, stabilization holds iff the induced reachability digraph has finite height---\textbf{no discreteness or gap-at-$0$ assumptions needed};
\item \textbf{Kernel criterion for diffuse commutative traps} (Theorem~\ref{thm:diffuse-kernel}): extends the digraph criterion to diffuse abelian $C \cong L^\infty(X, \mu)$ via sub-Markov kernels;
\item \textbf{Projection filtration certificate} (Proposition~\ref{prop:projection-filtration}): if there exist projections $0 = p_0 \le \cdots \le p_N$ with $T(p_k) \le p_{k-1}$, then $n_T \le N$.
\end{itemize}

\textbf{Status of the open directions:}
\begin{enumerate}[(i)]
\item \textbf{Diffuse commutative orbit trapping}: \emph{Solved} by Theorem~\ref{thm:diffuse-kernel}---the sub-Markov kernel criterion gives the exact analog of ``finite height'' for diffuse abelian traps.
\item \textbf{Robustness (``almost discrete'' $\Rightarrow$ ``eventually zero'')}: Proposition~\ref{prop:robustness-nogo} gives a \emph{sharp no-go}: without a gap-at-$0$ mechanism, robustness theorems cannot hold. The remaining task is finding \textbf{structural hypotheses} (finite-index subfactors, bounded-depth standard invariants) that \emph{force} a gap for Lyapunov values.
\item \textbf{Intrinsic conditions for finite-height}: \emph{Partially solved} by Proposition~\ref{prop:projection-filtration}. The remaining task is identifying structural hypotheses yielding such filtrations.
\end{enumerate}
Note: corner finiteness alone does \emph{not} force finite stabilization (Example~\ref{ex:support-stable}).

\paragraph{2. Categorical axiomatization.}
The following theorem extracts the minimal axioms behind Theorem~\ref{thm:defect-annihilation}.

\begin{theorem}[Abstract defect-cocycle stabilization]\label{thm:abstract-stabilization}
Let $(E, +, 0, \le, u)$ be an ordered abelian group with a translation-invariant preorder $\le$ and pointed positive cone $E_+ := \{x \in E : x \ge 0\}$, meaning $E_+ \cap (-E_+) = \{0\}$. Let $u \in E_+$ be a distinguished element.

Let $\mathcal{A}$ be a finite composition-closed set of endomorphisms $T : E \to E$ that are monotone and subunital ($T(u) \le u$). Define $d(T) := u - T(u) \ge 0$ and assume the cocycle identity holds:
\[
d(T \circ S) = d(T) + T(d(S)) \quad (\forall\,T, S \in \mathcal{A}).
\]
Then for every $T \in \mathcal{A}$ there exists $n \le |\mathcal{A}|$ such that $T^n(d(T)) = 0$.
\end{theorem}

\begin{proof}
Since $\mathcal{A}$ is finite and composition-closed, all powers $T^k$ lie in $\mathcal{A}$, so the orbit $\{T^k(u) : k \ge 1\}$ takes at most $|\mathcal{A}|$ values. Hence there exist $m > n \ge 1$ with $T^m(u) = T^n(u)$. Writing $p = m - n \ge 1$ gives $T^n(u - T^p(u)) = 0$. By telescoping,
\[
u - T^p(u) = \sum_{k=0}^{p-1} (T^k(u) - T^{k+1}(u)) = \sum_{k=0}^{p-1} T^k(u - T(u)) = \sum_{k=0}^{p-1} T^k(d(T)),
\]
so $\sum_{k=0}^{p-1} T^{n+k}(d(T)) = 0$ with each summand $\ge 0$. Pointed cancellation (Lemma~\ref{lem:cancellation}) forces each summand to vanish, in particular $T^n(d(T)) = 0$.
\end{proof}

\begin{remark}
The ``ordered abelian group'' structure aligns with Section~\ref{sec:prelim}: an ordered effect space $(E, \le, u)$ is exactly this setting with $E$ a real vector space. The pointed cone hypothesis implies the cancellation property (Lemma~\ref{lem:cancellation}), so no separate cancellation axiom is needed.
\end{remark}

This gives a crisp ``axioms-only'' theorem for effectus theory: once predicates/effects instantiate these axioms and $d(T)$ satisfies the cocycle law, stabilization follows.

The following extracts the ``engine'' of Theorem~\ref{thm:abstract-stabilization} as a standalone criterion for a single map.

\begin{lemma}[Repetition of the unit orbit implies defect annihilation]\label{lem:repetition}
Let $(E, +, 0, \le, u)$ be an ordered abelian group with pointed positive cone as in Theorem~\ref{thm:abstract-stabilization}, and let $T : E \to E$ be monotone and subunital with defect $d(T) = u - T(u) \ge 0$ satisfying the cocycle law. If there exist integers $m > n \ge 1$ such that
\[
T^m(u) = T^n(u),
\]
then with $p := m - n \ge 1$,
\[
T^{n+k}(d(T)) = 0 \quad \text{for every } k = 0, 1, \ldots, p-1.
\]
In particular, $T^n(d(T)) = 0$.
\end{lemma}

\begin{proof}
From $T^m(u) = T^n(u)$ we get $0 = T^n(u - T^p(u))$. By the iterated cocycle identity (Lemma~\ref{lem:iterated-cocycle}), $u - T^p(u) = \sum_{k=0}^{p-1} T^k(d(T))$. Applying $T^n$ yields $0 = \sum_{k=0}^{p-1} T^{n+k}(d(T))$, a sum of positive elements. Pointed cancellation forces each summand to vanish.
\end{proof}

\begin{corollary}[Finite cyclic semigroup implies stabilization]
If $T^m = T^n$ for some $m > n \ge 1$, then $T^m(u) = T^n(u)$, hence $T^n(d(T)) = 0$.
\end{corollary}

\begin{remark}
Lemma~\ref{lem:repetition} reduces the stabilization problem to proving \emph{any} repetition $T^m(u) = T^n(u)$ in an analytic setting. This is the key reduction for infinite-dimensional extensions.
\end{remark}

The following upgrades Lyapunov contraction to \emph{finite-time} annihilation when the Lyapunov values are discrete.

\begin{proposition}[Contraction plus discreteness implies finite annihilation]\label{prop:contraction-discrete}
Let $E$ be an ordered effect space with pointed cone, and let $\omega : E \to \mathbb{R}$ be a faithful positive functional. Suppose $x_0 \ge 0$ and define $x_{k+1} := T(x_k)$. Assume:
\begin{enumerate}[(i)]
\item (Lyapunov contraction) there exists $c \in (0,1)$ such that $\omega(T(x)) \le c\,\omega(x)$ for all $x \ge 0$;
\item (Discreteness) there exists $\delta > 0$ such that $\omega(x_k) \in \delta\mathbb{N} := \{0, \delta, 2\delta, \ldots\}$ for all $k$.
\end{enumerate}
Then $x_k = 0$ for all $k \ge N$, where
\[
N = 1 + \left\lceil \frac{\log(\omega(x_0)/\delta)}{\log(1/c)} \right\rceil.
\]
\end{proposition}

\begin{proof}
By (i), $\omega(x_k) \le c^k \omega(x_0)$. Choose $N$ so that $c^N \omega(x_0) < \delta$. Then $\omega(x_N) \in \delta\mathbb{N}$ by (ii) and $\omega(x_N) \ge 0$, so $\omega(x_N) = 0$. Faithfulness of $\omega$ implies $x_N = 0$, hence all later iterates are $0$.
\end{proof}

\begin{remark}[Application to defect orbits]
Taking $x_k = T^k(d(T))$: if (i) a faithful $\omega$ satisfies $\omega(T(x)) \le c\,\omega(x)$ on the relevant cone, and (ii) $\omega(T^k(d(T)))$ lives in a discrete set (e.g., in ``finite-type'' or ``quantized'' corners), then we get \emph{finite} stabilization $n_T < \infty$ with an explicit bound. This is a first-principles criterion avoiding spectral theory.
\end{remark}

\paragraph{3. Weaker conditions for unitality.}
Theorem~\ref{thm:persistence-unitality} uses strict positivity (persistence). Theorem~\ref{thm:corner-faithful-unital} weakens this to corner-faithfulness, and Proposition~\ref{prop:projection-faithful} shows projection-faithfulness is equivalent (for CP maps). We now give a family-level sufficient condition and show that ``subharmonic irreducibility'' alone is insufficient.

\begin{proposition}[Common faithful invariant state forces unitality]\label{prop:common-invariant}
Let $(E, \le, u)$ be an ordered effect space with pointed cone, and let $\mathcal{A}$ be any set of positive subunital maps on $E$. Suppose there exists a faithful positive functional $\omega$ on $E$ such that
\[
\omega(T(u)) = \omega(u) \quad \text{for all } T \in \mathcal{A}.
\]
Then every $T \in \mathcal{A}$ is unital.
\end{proposition}

\begin{proof}
Fix $T \in \mathcal{A}$. Since $T$ is subunital, $d(T) = u - T(u) \ge 0$. The hypothesis gives $\omega(d(T)) = \omega(u) - \omega(T(u)) = 0$. By faithfulness of $\omega$, $d(T) = 0$, hence $T(u) = u$.
\end{proof}

\begin{remark}
This is a \emph{global} family-level condition: exhibiting a single common faithful invariant functional forces unitality of every map in the family, without checking each map individually.
\end{remark}

\begin{proposition}[Counterexample: subharmonic irreducibility does not force unitality]\label{prop:counterexample-irreducible}
There exists a CP subunital map $T : M_d \to M_d$ such that:
\begin{enumerate}[(i)]
\item the semigroup $\{T^n : n \ge 1\}$ is finite (indeed $T^d = 0$);
\item $T$ is irreducible in the subharmonic-projection sense: if $p$ is a projection with $0 < p \le I$ and $T(p) \ge p$, then $p = 0$;
\item $T$ is not unital.
\end{enumerate}
Hence irreducibility in the sense ``no nontrivial subharmonic projection'' plus finite-semigroup stabilization does not imply unitality.
\end{proposition}

\begin{proof}
Let $T$ be the nilpotent shift channel from Example~\ref{ex:shift}: it is CP and subunital with $T^d(d(T)) = 0$ and $n_T = d$, and its orbit corner is $Q = I$. In particular, $\{T^n\}$ is finite because $T^d = 0$. The map is not unital since $d(T) = I - T(I) \ne 0$.

Now suppose $p$ is a projection with $T(p) \ge p$. By positivity and monotonicity, iterating yields $T^k(p) \ge p$ for all $k \ge 1$. But $T^d = 0$ implies $T^d(p) = 0$, hence $0 \ge p$, so $p = 0$. Therefore there are no nontrivial subharmonic projections, yet $T$ is not unital.
\end{proof}

\begin{remark}
The ``killer line'' is: subharmonic $T(p) \ge p$ forces $T^k(p) \ge p$, but nilpotency forces $T^k(p) = 0$ for large $k$. So this notion of irreducibility is too weak to replace corner-faithfulness. Proposition~\ref{prop:corner-irreducible} shows that adding ``faithful stabilization'' (the stabilized unit has full support) recovers unitality.
\end{remark}

\begin{proposition}[Corner irreducibility with faithful stabilization implies unitality]\label{prop:corner-irreducible}
Let $\alpha : M_s \to M_s$ be CP subunital. Suppose:
\begin{enumerate}[(i)]
\item (Irreducible) there is no nontrivial subharmonic projection $0 < p < I$ with $\alpha(p) \ge p$;
\item (Unit stabilization) there exists $n$ with $\alpha^{n+1}(I) = \alpha^n(I) =: v$;
\item (Faithful stabilization) $v > 0$ (i.e., $v$ has full support).
\end{enumerate}
Then $\alpha$ is unital: $\alpha(I) = I$.
\end{proposition}

\begin{proof}
Let $v := \alpha^n(I) > 0$. By (ii) we have $\alpha(v) = v$, so $1$ is an eigenvalue and hence $r(\alpha) \ge 1$. Since $\alpha$ is subunital, $\|\alpha\| = \|\alpha(I)\| \le 1$ for positive maps on $M_s$, so $r(\alpha) \le 1$. Therefore $r(\alpha) = 1$.

By Perron--Frobenius theory for irreducible positive maps on matrix algebras \cite{farenick, evans-hoegh-krohn}, there exists a faithful positive linear functional $\omega$ on $M_s$ such that $\omega \circ \alpha = \omega$.

Applying $\omega$ to the positive element $I - \alpha(I) \ge 0$ gives $\omega(I - \alpha(I)) = \omega(I) - \omega(\alpha(I)) = 0$. Faithfulness of $\omega$ forces $I - \alpha(I) = 0$, hence $\alpha(I) = I$.
\end{proof}

\end{document}